\documentclass[a4paper,10pt]{article}
\usepackage[utf8]{inputenc}
\usepackage[T1]{fontenc}
\usepackage[english]{babel}
\usepackage{amsmath}
\usepackage{amsthm}
\usepackage{amsfonts}
\usepackage{amssymb}
\usepackage{listings}
\usepackage{graphicx}
\usepackage{mathrsfs}
\usepackage{hyperref}
\usepackage{float}

\usepackage{algorithm}
\usepackage{algorithmic}

\usepackage{subfig}
\usepackage{stmaryrd}
\usepackage{bm}
\usepackage{tikz}
\usepackage{braket}
\usepackage{wasysym}
\usepackage{wrapfig}
\usepackage[framemethod=TikZ]{mdframed}
\usepackage{xcolor}
\usepackage{pstricks}
\usepackage{faktor}

\usepackage[top=2.00cm,bottom=3.00cm,left=2.00cm,right=2.00cm]{geometry} 

\definecolor{brightpink}{rgb}{1.0, 0.0, 0.5}

\newcommand{\revise}[1]{{{\color{black} #1}}}

\newcommand{\mc}{\mathscr}
\newcommand{\ma}{\mathcal}
\newcommand{\f}{\mathbb}

\newcommand{\cu}{\subseteq}
\newcommand{\wt}{\widetilde}

\newcommand{\ve}{\varepsilon}

\DeclareMathOperator{\diag}{diag}

\DeclareMathOperator{\rk}{rk}
\DeclareMathOperator{\conv}{conv}
\DeclareMathOperator{\argmax}{argmax}
\DeclareMathOperator{\nnz}{nnz}

\theoremstyle{definition}

\theoremstyle{plain}
\newtheorem{theorem}{Theorem}

\newtheorem{lemma}{Lemma}
\newtheorem{remark}{Remark}

\title{On the Robustness of the Successive Projection Algorithm} 

\author{Giovanni Barbarino\thanks{GB is member of the Research Group GNCS (Gruppo Nazionale per il Calcolo Scientifico) of INdAM (Istituto Nazionale di Alta Matematica).} \qquad  Nicolas Gillis\thanks{Email: \{giovanni.barbarino, nicolas.gillis\}@umons.ac.be. Authors acknowledge the support
 by the European Union (ERC consolidator, eLinoR, no 101085607).  
 } \\ 
University of Mons, Rue de Houdain 9, 7000 Mons, Belgium   
	}

\date{}

\begin{document}

\maketitle 

\begin{abstract} 
The successive projection algorithm (SPA) is a workhorse algorithm to learn the $r$ vertices of the convex hull of a set of $(r-1)$-dimensional data  points, a.k.a.\ a latent simplex, which has numerous applications in data science. In this paper, we revisit  the robustness to noise of SPA and several of its variants. In particular, when $r \geq 3$, we prove the tightness of the existing error bounds for SPA and for two more robust preconditioned variants of SPA. We also provide significantly improved error bounds for SPA, by a factor proportional to the conditioning of the $r$ vertices, in two special cases: for the first extracted vertex, and when $r \leq 2$. We then provide further improvements for the error bounds of a translated version of SPA proposed by Arora et al.\ (``A practical algorithm for topic modeling with provable guarantees'', ICML, 2013) in two special cases: for the first two extracted vertices, and when $r \leq 3$. Finally, we propose a new more robust variant of SPA that first shifts and lifts the data points in order to minimize the conditioning of the problem. We illustrate our results on synthetic data.     
\end{abstract}

\section{Introduction} 

The problem of finding a latent simplex via the observation of noisy data points generated within that simplex is a fundamental problem in signal processing, data analysis, and machine learning. It finds applications in 
chemometrics~\cite{Araujo01, pasquini2018near}, 
hyperspectral imaging~\cite{ma2014signal}, 
audio source separation~\cite{virtanen2007monaural}, 
topic modeling~\cite{arora2013practical, bakshilearning}, and community detection~\cite{mao2021estimating, jin2024mixed, agterberg2024estimating}, to cite a few. 
The problem is posed as follows: we observe samples 
$$x_j = Wh_j + N(:,j)  \in \mathbb{R}^{m} \text{ for }  j=1,2,\dots,n,
$$ 
where 
\begin{itemize}
    \item The column of $W \in \mathbb{R}^{m \times r}$  are $r$ vertices, to be found, $\{w_k \in \mathbb{R}^{m}\}_{k=1}^r$. 

    \item The weights $h_j \in \mathbb{R}^r$ are nonnegative and sum to one, that is, they belong to the probability simplex, $h_j \in \Delta = \{ y \in \mathbb{R}^r \ | \ y \geq 0, \sum_{i=1}^r y_i = 1  \}$ for all $j$.   
    We will denote $H = [h_1, h_2, \dots, h_n] \in \mathbb{R}^{r \times n}$. 

    \item The vector $N(:,j) \in \mathbb{R}^{m}$ models the noise and model misfit. 
\end{itemize}
Hence the samples, $x_j$'s, are contained in a latent simplex\footnote{A simplex is a $d$-dimensional poltyope with $d+1$ vertices, e.g., the probability simplex. For $\conv(W) \subset \mathbb{R}^m$ to be a simplex embedded in an $m$-dimensional space, the affine hull of the columns of $W$, that is, $\{ Wx \ | \ x^\top  e = 1 \}$, must have dimension $r-1$.}, $\conv(W) = \{ x \ | \ x = Wh, h \in \Delta^r \}$, up to the noise $N(:,j)$, and the goal is to learn $W$ given $X = WH + N$. 
This problem is sometimes referred to as ``learning a latent simplex''~\cite{bakshilearning} or ``simplex-structured matrix factorization'' (SSMF)~\cite{abdolali2021simplex}, or ``probabilistic simplex component analysis''~\cite{wu2021probabilistic}; we will use SSMF in this paper. 
When $W \geq 0$, this is a nonnegative matrix factorization (NMF) problem~\cite{lee1999learning}. 

Many variants of SSMF have been studied, making various assumptions on $W$, $H$ and $N$. 
An assumption that has proved particularly powerful is \emph{separability} which requires that for each vertex (that is, each column of $W$), there exists at least one data point close to that vertex. Mathematically: for all $k=1,2,\dots,r$, there exists $j$ such that $X(:,j) = W(:,k) + N(:,j)$. This is equivalent to require that $H$  contains the identity matrix as a submatrix since the previous equation implies $H(:,j) = e_k$, where $e_k$ the $k$th unit vector. If $H$ satisfies this condition, it is said to be a separable matrix.  
The separability assumption is also known as the pure-pixel assumption in hyperspectral imaging~\cite{ma2014signal}, or the anchor-word assumption in topic modeling~\cite{arora2013practical}. In this paper, we will refer to the SSMF problem under the separability assumption as \emph{separable SSMF}. 

Countless algorithms have been developed for separable SSMF; see, e.g., \cite[Chap.~7]{gillis2020nonnegative} and the references therein. Among them, a workhorse algorithm is the successive projection algorithm (SPA)~\cite{Araujo01} which is described in Algorithm~\ref{algo:spa}. 
Geometrically, when $\rk(W) = r$, $\conv(W)$ is an $(r-1)$-dimensional simplex with $r$ vertices embedded in a higher $m$-dimensional space. 
The first step of SPA extracts the point with the largest $\ell_2$ norm. For any polytope (such as a simplex), such a point is always a vertex. Then SPA projects all data points on the orthogonal complement of that extracted point (so that this point is projected onto zero), reducing the dimension of the simplex by one, and continuing the process $(r-1)$ times. 

\algsetup{indent=2em}
\begin{algorithm}[ht!]
\caption{Successive Projection Algorithm (SPA)~\cite{Araujo01}  \label{algo:spa}}
\begin{algorithmic}[1] 
\REQUIRE The matrix $X = WH + N$ where $H \geq 0$, $H^\top  e \leq e$, $H$ is separable, $W$ is full column rank, 
the number~$r$ of columns to extract.  
 
\ENSURE Set of $r$ indices $\mathcal{J}$ 
such that $\tilde{X}(:,\mathcal{J}) \approx W$ (up to permutation). 

    \medskip  
		
\STATE Let $R = X$, $\mathcal{J} = \{\}$.  \vspace{0.1cm} 
\FOR {$k = 1, \dots r$}   \vspace{0.1cm}  
\STATE $p \in \argmax_j \| R(:,j) \|$.     
\qquad \quad \textit{\% Selection step} \vspace{0.1cm}  
\STATE $\mathcal{J} \leftarrow \mathcal{J} \cup \{p\}$. \vspace{0.1cm} 
\STATE $R \leftarrow \left(I-\frac{{R(:,p)} R(:,p)^\top }{\|{R(:,p)}\|^2}\right)R$. \hspace{0.4cm} \textit{\% Projection step}  
\vspace{0.1cm}    
\ENDFOR 
\end{algorithmic}  
\end{algorithm} 

\begin{remark}[Implementation of SPA] \label{rem:spa}
    SPA should not be implemented as in Algorithm~\ref{algo:spa}, this would be particularly ineffective for large sparse matrices, \revise{as the residual would become quickly dense (especially if the first extracted columns is dense).} It can be implemented in $\mathcal{O}(r \nnz(X))$ operations (essentially, $r$ matrix-vector products), where $\nnz(X)$ is the number of non-zero entries of $X$, using the fact that $\| (I - u u^\top ) y\|^2 = \|y\|^2 - (u^\top  y)^2$ where $u$ has unit norm and $y$ is any vector~\cite{gillis2013fast}. See \url{https://gitlab.com/ngillis/robustSPA} for a MATLAB implementation.
\end{remark}

SPA has a long history, as the index set of columns it extracts corresponds to that of QR with column pivoting~\cite{Busi1965QR}. However, it was introduced in the context of separable SSMF by Ara\'ujo et al.~\cite{Araujo01}, and rediscovered many times under different names; see~\cite[p.~229]{gillis2020nonnegative} for an historical overview. 

Besides being simple and computationally very cheap, SPA was one of the first algorithm for separable SSMF to be proved to be robust to noise~\cite{gillis2013fast}; see the next section for more details. Let us recall this result. 
\begin{theorem}~\cite[Theorem 3]{gillis2013fast} \label{th:spanoise}
Let $X = WH + N \in \mathbb{R}^{m \times n}$ where $H \in \mathbb{R}^{r \times n}_+$ is a separable matrix satisfying $H^\top  e \leq e$, with $e$ the vector of all one of appropriate dimension, and let 
\begin{equation} 
\varepsilon = \max_j \| N(:,j)\| 
\leq \mathcal{O}\left(  \frac{ \sigma_r(W)}{\sqrt{r} \mathcal K^2(W)} \right), 
\tag{Bound on the noise}
\end{equation}
where $\sigma_r(W)$ is the $r$th singular value of $W$, and $\mathcal K(W) = \frac{K(W)}{\sigma_r(W)}$ with $K(W) = \max_k \|W(:,k)\|$ is a measure of the conditioning\footnote{Note that $\mathcal K(W)$ is smaller than the condition number of $W$, $\frac{\sigma_1(W)}{\sigma_r(W)}$. In fact, recall that $\sigma_1(W) = \|W\| \ge K(W) \ge \sigma_1(W)/\sqrt r$, where $r$ is the number of columns in $W$.} of $W$. 
Then SPA (Algorithm~\ref{algo:spa}) returns a set of indices $\mathcal{J}$ such that 
\begin{equation}
\max_{1 \leq k \leq r} \min_{j} \|W(:,k) - X(:,\mathcal{J}_j) \| 
\leq 
\mathcal{O}\left(  \varepsilon \mathcal K^2(W) \right). \tag{Error bound}
\end{equation}
\end{theorem} 


Note that it is not required that $H$ is column stochastic, that is, $H^\top  e = e$, but only that $H^\top  e \leq e$; we will discuss this in Section~\ref{sec:previousw}.

\paragraph{Outline and contribution} In this paper, we revisit the robustness to noise of SPA and some of its variants, and provide a much more complete picture than that of Theorem~\ref{th:spanoise}.  
After summarizing the current knowledge about the robustness to noise of SPA in Section~\ref{sec:previousw}, our contributions are as follows. 
\begin{itemize}
    \item Section~\ref{sec:singlestep}. We reanalyze the selection step of SPA, showing that the error bound of SPA can be improved, namely in  $\mathcal{O}\left(\varepsilon  \mathcal K(W) \right)$ instead of $\mathcal{O}\left(\varepsilon  \mathcal K^2(W) \right)$ for the first step of SPA (Theorem~\ref{theo:single_step_SPA}).  We then show that for a rank-2 separable SSMF problem, SPA achieves the error bound in  $\mathcal{O}\left(\varepsilon  \mathcal K(W) \right)$     
with a bound on the noise given by $\varepsilon \leq \mathcal{O}\left(\frac{\sigma_r(W)}{ \mathcal K(W)} \right)$, a significant  improvement compared to SPA, of a factor $\mathcal K(W)$ for both terms (Theorem~\ref{theo:rank2_SPA_Q}).  

    \item Section~\ref{sec:T-SPA}.  For separable SSMF, the first projection step of SPA can be replaced by a translation, as done in~\cite{arora2013practical}, which we refer to as translated SPA (T-SPA). 
    Adapting the results from Section~\ref{sec:singlestep}, we show that the first two steps of T-SPA have an error bound in $\mathcal{O}\left(\varepsilon  \mathcal K(W) \right)$ (Theorem~\ref{theo:firsttwo_T-SPA}). Moreover, for a  rank-3 separable SSMF problem, T-SPA  achieves the error bound in $\mathcal{O}\left(\varepsilon  \mathcal K(W) \right)$   
with a bound on the noise given by $\varepsilon \leq \mathcal{O}\left(\frac{\sigma_r(W)}{ \mathcal K(W)} \right)$, the same improvement as for SPA in Section~\ref{sec:singlestep} but for $r=3$  (Theorem~\ref{theo:rank3_T-SPA}).  
    
    \item Section~\ref{sec:tighT-SPA}. For $r \geq 3$, we show that the bound of Theorem~\ref{th:spanoise} from~\cite{gillis2013fast} with error bound in $\mathcal{O}\left(\varepsilon  \mathcal K^2(W) \right)$ is tight (Theorem~\ref{theo:tightSPA}). We do this by providing a family of data matrices that achieve this worst-case bound already at the second step of SPA.  
    This answers an open question about the tightness of the SPA error bound of Theorem~\ref{th:spanoise}~\cite[p.~231]{gillis2020nonnegative}. 

    \item Section~\ref{sec:precond}. SPA can be preconditioned in various ways to obtain error bounds in $\mathcal{O}\left(\varepsilon  \mathcal K(W) \right)$~\cite{gillis2015semidefinite, gillis2015enhancing} (see Section~\ref{sec:draw2} for more details).     
    We show that these error bounds are tight for two important preconditioners: SPA itself~\cite{gillis2015enhancing} (Theorem~\ref{theo:tightSPA2}), and minimum-volume ellipsoid (MVIE)~\cite{gillis2015semidefinite} (Theorem~\ref{theo:tightMVESPA}).

    \item Section~\ref{sec:preprocess}. We discuss a way to preprocess the data, by translating and lifting the data points in a higher dimensional space of dimension $m+1$, 
    allowing us to improve robustness of SPA be reducing the conditioning of $W$. 

    \item Section~\ref{sec:numexp}. We provide numerical experiments to compare the various SPA variants and how they compare on synthetic data, allowing us to illustrate our theoretical findings. In particular, our experiments show that our proposed SPA variant outperforms the others in adversarial settings. 
\end{itemize}

\paragraph{Notation} 
Given a vector $x \in \mathbb{R}^m$, we denote $\|x\|$ its $\ell_2$ norm. 
Given a matrix 
$X \in \mathbb{R}^{m \times n}$, we denote $X^\top $ its transpose, 
$X(:,j)$ its $j$th column, 
$X(i,:)$ its $i$th row, 
$X(i,j)$ its entry as position $(i,j)$, 
$\|X\| = \sigma_{\max}(X)$ its $\ell_2$ norm which is equal to its largest singular value, 
$\|X\|_F^2 = \sum_{i,j} X(i,j)^2$ its squared Frobenius norm, 
$\sigma_r(X)$ its $r$th singular value,  
$\sigma_{\min}(X) = \sigma_{\min(m,n)}(X)$ its smallest singular value, $\rk(X)$ its rank, 
$K(X)$ $=$ \mbox{$\max_j \{\|X(:,j)\|\}$}, and 
$\mathcal K(X) = K(X)/\sigma_{\min}(X)$. 
For $k> \min\{m,n\}$, by convention,  $\sigma_k(X)=0$. 
Note that $$
 \sigma_{n} (X)\le   \min_i\{\|X(:,i)\|\}\le K(X)\le 
 \|X\|, \quad 
\sigma_{n}(X) \le \|P_{X(:,i)}(X(:,j))\| \le \|X(:,i)-X(:,j)\|, $$ 
and
$\sigma_{n}(X) \sqrt 2\le \|X(:,i)-X(:,j)\|$ for every $i\ne j$, where $P_u(v)$ is the projection of $v$ on the subspace orthogonal to $u$. 
We denote $e_k$ the $k$th unit vector, $e$ the vector of all ones of appropriate dimension, $I_r$ the identity matrix of dimension $r$. 
The set $\mathbb{R}^{m \times n}_+$ denotes the $m$-by-$n$ component-wise nonnegative matrices. A matrix $H  \in \mathbb{R}^{r \times n}$ has stochastic columns if $H \geq 0$ and $H^\top  e = e$. It is separable if $H(:,\revise{\mathcal{J}}) = I_r$ for some index set \revise{$\mathcal J$} of size $r$. 

Given $W \in \mathbb{R}^{m \times r}$, we may also denote its columns by $\{w_j\}_{j=1}^r$. 
The convex hull generated by the columns of $W$ is denoted $\conv(W) = \conv(w_1,w_2,\dots,w_r) = \{ Wh \ | \ e^\top  h = 1, h \geq 0\}$. 

The notation $A \succ 0$ means that $A$ is a square symmetric positive definite matrix.  

\revise{In the results, we use the notation $\dots  \le \mathcal O(f(\bf x))$  to indicate that a quantity is bounded by $f(\bf x)$ up to an absolute multiplicative constant $C$ that does not depend on any of the parameters or variables.  }

\section{SPA: what do we know?} \label{sec:previousw}

 As described in the previous section, SPA is fast and robust to noise. However, it has several drawbacks, including the following ones: 
 \begin{enumerate}

     \item It can only extract as many as $\rk(W)$ vertices. 
     For example, if the columns of $W$ live in a an $m$-dimensional space with $\rk(W) \leq m < r$, e.g., a triangle in the plane ($m=\rk(W)=2$, $r=3$), SPA can only extract $\rk(W)$ of the $r$ vertices.

     \item The error bound are rather poor, depending on the squared of the condition number of $W$. Moreover, it is not known whether this bound is tight~\cite[p.~231]{gillis2020nonnegative}.

     \item SPA is sensitive to outliers. 
     
 \end{enumerate}
    
Researchers have proposed many improvements and modifications of SPA to alleviate some of these drawbacks. 
Let us highlight some of these results; in particular the ones we will discuss in this paper. 

\subsection{Drawback 1: rank-deficient case}  \label{sec:draw1}

How can we adapt SPA to handle the case when $\rk(W) < r$? 
 Assuming that $H$ is column stochastic, that is, $H^\top  e = e$, 
         Arora et al.~\cite{arora2013practical} replace the first projection step of SPA by a translation step; see Algorithm~\ref{algo:T-SPA}. Note that Arora et al.~\cite{arora2013practical}  did not present their algorithm exactly in these terms (they worked with affine spaces) but Algorithm~\ref{algo:T-SPA} is equivalent. 
         \algsetup{indent=2em}
\begin{algorithm}[ht!]
\caption{Translated SPA (T-SPA)~\cite{arora2013practical}  \label{algo:T-SPA}}
\begin{algorithmic}[1] 
\REQUIRE The matrix $X = WH + N$ where $H \geq 0$, $H^\top  e = e$, $H$ is separable, $W$ is full column rank, 
the number~$r$ of columns to extract.  
 
\ENSURE Set of $r$ indices $\mathcal{J}$ 
such that ${X}(:,\mathcal{J}) \approx W$ (up to permutation). 

    \medskip  
		
\STATE Let $R = {X}$, $\mathcal{J} = \{\}$.  \vspace{0.1cm} 
\FOR {$k = 1, \dots r$}   \vspace{0.1cm}  
\STATE $p \in \argmax_j \| R(:,j) \|$.     
\hspace{1.65cm} \textit{\% Selection step} \vspace{0.1cm}  
\STATE $\mathcal{J} = \mathcal{J} \cup \{p\}$. \vspace{0.1cm} 
\IF{$k = 1$}
\STATE $R \leftarrow R - X(:,p) \ e^\top $. \hspace{1.35cm} \textit{\% Translation step} 
\ELSE 
\STATE $R \leftarrow \left(I-\frac{{R(:,p)} R(:,p)^\top }{\|{R(:,p)}\|^2}\right)R$. \hspace{0.45cm} \textit{\% Projection step} 
\ENDIF 
\vspace{0.1cm}    
\ENDFOR 
\end{algorithmic}  
\end{algorithm} 

\begin{remark}[Implementation of T-SPA] \label{rem:tspa}
As for SPA (see Remark~\ref{rem:spa}), 
    T-SPA can be implemented in $\mathcal{O}(r \nnz(X))$ operations (essentially, $r$ matrix-vector products). See \url{https://gitlab.com/ngillis/robustSPA} for a MATLAB implementation. 
\end{remark}

Replacing the projection step by a translation at the first iteration has two advantages: 
\begin{itemize}
    \item It allows T-SPA to extract $\rk(W) +1$ vertices. 
    The simplest examples are triangles in the plane, e.g., 
    \begin{equation*} 
      W = \left( \begin{array}{ccc}
   -1  & 1 & 0 \\
    -1 &  0 & 1
\end{array}
\right),    
    \end{equation*}
    for which SPA can only extract 2 out of the 3 vertices, while T-SPA can extract all of them.

    \item Arora et al.~\cite{arora2013practical} proved an error bound in $\mathcal{O}\left(\varepsilon  \mathcal K^2(W_t) \right)$ for T-SPA, where $W_t \in \mathbb{R}^{m \times (r-1)}$ is the translated version of $W$ where the zero column is discarded, that is,  
    \[
    W_t = [w_1-w_p,\dots,w_{p-1}-w_p,w_{p+1}-w_p,\dots,w_r-w_p],  
    \] 
    where $w_p$ is the first extracted column by T-SPA, that is, the column of $W$ with the largest $\ell_2$ norm. 
    Note that this bound essentially follows that of~\cite{gillis2013fast} after translation. 
Interestingly, and this has never been pointed out in the literature as far as we know, the matrix $W_t$ can be significantly better conditioned than $W$, even when $\rk(W) = r$. 
For example, let 
    \begin{equation*} 
      W = \left( \begin{array}{ccc}
   -1  & 1 & 0 \\
    -1 &  0 & 1 \\ 
    \delta & \delta & \delta 
\end{array}
\right) \text{ for } \delta \ll 1,     \text{ with } 
 W_t = \left( \begin{array}{ccc}
    2 & 1 \\
      1 & 2 \\ 
     0 & 0 
\end{array}
\right), 
    \end{equation*} 
where $\mathcal K(W) = \mathcal O(1/\delta)$ while $\mathcal K(W_t) = \sqrt{5}$.  

Note that we will always have $\mathcal K(W_t) \leq 2 \mathcal K(W)$. In fact, $K(W_t) \leq 2 K(W)$ since $\|w_i - w_j\| \leq \|w_i \| + \| w_j\|$,  and $\sigma_{r-1}(W_t) \geq \sigma_{r}(W)$ since $W_t$ is a rank-one correction of $W$ (by the interlacing singular value theorem).   
    \end{itemize}
    
Another possibility to extract $\rk(W) +1$ vertices when the affine hull of the columns of $W$ contains the origin is to replace $X$ by $\binom{X}{c e^\top }$ where $c > 0$ is a constant~\cite{mao2021estimating, jin2024mixed}. This amounts to lift the data points in one dimension higher, and leads to an equivalent SSMF problem: for $H^\top  e = e$, 
\begin{equation} \label{eq:lift}
X = WH \; \iff \; \binom{X}{c e^\top } = \binom{W}{c e^\top } H. 
\end{equation} 
This prevents the affine hull of $\binom{W}{c e^\top }$ to contain the origin. We will discuss in Section~\ref{sec:preprocess} how to perform this lifting (and combining it with a translation) in order to minimize the conditioning of $\binom{W}{c e^\top }$, and hence improve the performance of SPA.

The two methods described above can only extract the vertices of a simplex, that is, of an $(r-1)$-dimensional poltyope with $r$ vertices. 
In order to extract more vertices, 
one can take the nonnegativity of $H$ into account in the projection step, leading to the successive nonnegative projection algorithm (SNPA). SNPA can extract all the vertices of $\conv(W)$, regardless of the dimensions; see~\cite{gillis2014successive} for more details.
However, when $\rk(W)=r$, SPA and SNPA have the same error bounds, while SPA is significantly faster. 
Other approaches that allow to extract more than $\rk(W)+1$ vertices are based on convex formulations~\cite{Esser2012convex, elhamifar2012see, recht2012factoring, gillis2013robustness, gillis2014robust, mizutani2022refinement}. 
We will not discuss SNPA and the convex relaxations further in this paper, nor the tightness of their error bounds, which is a topic of further research.

\subsection{Drawback 2: sensitivity to noise} \label{sec:draw2}

A standard preprocessing step in the literature to improve the robustness to noise of SPA is to replace $X$ with its best rank-$r$ approximation using the truncated singular value decomposition (SVD); see, e.g., \cite{ma2014signal, jin2024improved}. 
It is also possible to precondition the input matrix to improve the performance of SPA in the presence of noise. 
The intuition behind preconditioning is to try to estimate the left inverse of $W$, $W^\dagger$, from the data, then multiply $X$ by $W^\dagger$, and apply SPA on $W^\dagger X = H + W^\dagger N$. In the preconditioned problem, $W$ becomes the identity which is perfectly conditioned. 
In this paper, we will focus on two preconditionings because they offer the most robust variants of SPA:   
\begin{itemize}
    \item Minimum-volume ellipsoid (MVE)~\cite{gillis2015semidefinite}: find the $r$-dimensional MVE centered at the origin, $\{ x \ | \ x^\top  A x \leq 1 \}$ where $A \succ 0$,  that contains all data point. The MVE can be computed via semidefinite programming. Then multiply $X$ by the left inverse of $A$.  
This allows one to reduce the bound on the noise in SPA  to $\varepsilon \leq \mathcal{O}\left(  \frac{ \sigma_r(W)}{ r \sqrt{r}} \right)$, with a factor of improvement of $\frac{\mathcal K^2(W)}{r}$, and reduce the error bound to $\mathcal{O}\left(  \varepsilon \mathcal K(W) \right)$, with a factor of improvement of $\mathcal K(W)$. This approach is referred to as MVE-SPA. Note that a similar approach was proposed in~\cite{mizutani2014ellipsoidal}. 

Note that MVE-SPA needs to work with an $r$-by-$n$ input matrix, and hence a dimensionality reduction preprocessing is necessary, such as the truncated SVD.

    \item SPA~\cite{gillis2015enhancing}: Estimate $W^\dagger$ by the left inverse of the SPA solution, $X(:,\mathcal{J})$. This strategy might appear a bit odd, preconditioning SPA with itself. It has the same bound on the noise as SPA, but allows to reduce the error bound to  $\mathcal{O}\left(  \varepsilon \mathcal K(W) \right)$, while providing significantly better results in numerical experiments~\cite{gillis2015enhancing} (see also  Section~\ref{sec:numexp}). We will denote SPA  preconditioned with SPA as SPA$^2$.  
\end{itemize}
We will show in Section~\ref{sec:precond} that the error bounds for these two preconditionings are also tight.

Another way to improve the error bound of SPA to $\mathcal{O}\left(  \varepsilon \mathcal K(W) \right)$ (but not the bound on the noise) was proposed in~\cite{arora2013practical}: after a first set of $r$ indices $\mathcal{J}$ is extracted by SPA, each element  $j \in \mathcal{J}$ is reviewed again and replaced with the index $p$ that maximizes the volume of $\conv(X(:,\mathcal{J} \backslash \{j\} \cup \{ p \})$.  
However, this algorithm, referred to as Fast Anchor Word (FAW) typically performs  worse than SPA$^2$~\cite{gillis2015enhancing}; see also Section~\ref{sec:numexp} for numerical experiments. Moreover, the computational cost of FAW is $r$ times that of SPA. Note that FAW uses a translation at the first step, that is, it relies on T-SPA. 

Another direction of research to improve robustness to noise is to consider randomized algorithms, such as vertex component analysis (VCA)~\cite{nascimento2005vertex} or randomized SPA~\cite{vuthanhrandSPA2022}. Such algorithms  generate a different solution at each run, from which one can pick the best one according to some criterion (e.g., the reconstruction error or the volume of $\conv(W)$). Discussing such algorithms is out of the scope of this paper. 



\subsection{Drawback 3: outliers} To improve performance in the presence of outliers, standard strategies include outlier removal, see, e.g., \cite{jin2024improved}, and robust low-rank approximations of the input matrix~\cite{candes2011robust}. 
Another dedicated strategy for SPA is to check, after the selection step, that the extracted column of $X$ can sufficiently reduce the error in the approximation (that is, $\min_{H} \| X - X(:,\mathcal J) H\|$ is sufficiently small)~\cite{gillis2019robustSPA}. 
A more recent strategy leverages the presence of multiple columns of $X$ close to that of $W$, allowing to estimate the columns of $W$ as the average or median of several columns of $X$; 
see~\cite{bhattacharyya2020finding, bakshilearning, nadisic2023smoothed}. 
We will not cover these approaches in this paper.

\section{Improved bounds for the first step of SPA and when $r=2$} \label{sec:singlestep}

In this section, we provide improved bounds for SPA: for the first step in Section~\ref{sec:firststepSPA} (Theorem~\ref{theo:single_step_SPA}), and when $r=2$ in Section~\ref{sec:SPAr2} (Theorem~\ref{theo:rank2_SPA_Q}). 

\revise{
\paragraph{Motivations} From a theoretical standpoint, understanding better the behavior of SPA, a workhorse separable NMF algorithm, is a mathematically interesting problem, even when $r=2,3$. 
From a practical standpoint, rank-2 NMF and/or SPA with $r=2$ have been used successfully to: 
 \begin{itemize}
     \item segment and analyze medical images~\cite{li2013hierarchical, sauwen2015hierarchical},  

    \item extract topics in large corpus of documents~\cite{kuang2013fast}, and 

     \item cluster pixels in hyperspectral images~\cite{gillis2014hierarchical}. 
     
 \end{itemize}

In analytical chemistry, where NMF is referred to as the self-modeling curve resolution (SMCR) problem, using $r=3$ is a standard setting~\cite{borgen1985extension, neymeyr2018set}. 
}

\subsection{First step of SPA} \label{sec:firststepSPA} 

The robustness proof of SPA in Theorem~\ref{th:spanoise} relies on using sequentially the following theorem that studies a single step of SPA (we report the bounds using the $\mathcal{O}$ notation to simplify the presentation\footnote{In~\cite{gillis2013fast}, authors consider any strongly convex and $L$-smooth function $f$ to be maximized, not only the $\ell_2$ norm. We focus in this paper on the $\ell_2$ norm which is the most used variant.}). 
\begin{theorem}\cite[Theorem~2]{gillis2013fast}
\label{theo:single_step_SPA_original}    
Let $W\in \f R^{m\times k}$, and $Q\in \f R^{m\times (r-k)}$ and $X = [W\,\, Q]H + N\in\f R^{m\times n}$ with $k\le r\le \min\{m,n\}$, where  $W$ is full rank, $K(N)< \varepsilon\le  \mathcal{O} \left( \frac{\sigma_k(W)}{  \mathcal K(W) } \right)$ and $H\in \f R_+^{r\times n}$ is separable with $H^\top  e \leq e$. 
If $K(Q) \le  \mathcal{O} \left( \mathcal K(W) \varepsilon \right)$, then the column of $X$ with the largest $\ell_2$ norm, denoted $x$, satisfies 
\[
\min_{1\leq i \leq r} \| x - w_i\| \le \mathcal{O} \left(\mathcal K^2(W)\varepsilon \right).
\]
\end{theorem}

In this section, we improve the bound of Theorem~\ref{theo:single_step_SPA_original}; see 
Theorem~\ref{theo:single_step_SPA} below which is very similar to Theorem~\ref{theo:single_step_SPA_original}, but  improves the error bound, namely from   $\mathcal{O}\left(\varepsilon  \mathcal K^2(W) \right)$ to $\mathcal{O}\left(\varepsilon  \mathcal K(W) \right)$. 
Note that it does not improve \cite[Theorem 2]{gillis2013fast} in terms of the  bound on the noise allowed by SPA which was already of the order of  $\varepsilon \leq \mathcal{O}\left(\frac{\sigma_r(W)}{ \mathcal K(W)} \right)$.


\begin{theorem}
\label{theo:single_step_SPA}    Let $W\in \f R^{m\times k}$,   $Q\in \f R^{m\times (r-k)}$ and $X = [W\,\, Q]H + N\in\f R^{m\times n}$ with $k\le \min\{m,r\}$, where  $W$ is full rank, 
$K(N)< \varepsilon\le \sigma_k(W)\mathcal K(W)^{-1}/8$, $H\in \f R_+^{r\times n}$ is separable and has stochastic columns. 
If $K(Q) \le  K(W)/2$, 
then for the column of $X$ with largest $\ell_2$ norm, denoted $x = X(:,p) = [W, Q]h_p + N(:,p)$ for some $p$, there exists a column $w_i$ of $W$ such that 
$\| x -w_i\| \le 33\mathcal K(W)\varepsilon$ 
and $h_p(i) \ge 1/2$.

\end{theorem}

Notice that under these hypotheses, $K(W)/2\ge \sigma_k(W)/2\ge 4\mathcal K(W)\ve$, so the upper bound on $K(Q)$ is consistent with the one in Theorem \ref{theo:single_step_SPA_original} and possibly less restrictive. 

To prove Theorem~\ref{theo:single_step_SPA}, we use a similar proof strategy as that of \cite[Theorem~1]{jin2024improved}, which is based on the geometry of the problem, as opposed to matrix norm inequalities used in~\cite{gillis2013fast}, as explained in~\cite{jin2024improved}.  
However,  \cite[Theorem~1]{jin2024improved} still  has a factor proportional to $\mathcal K^2(W)$ in the error bound, and hence its bound is asymptotically similar to that of Theorem~\ref{theo:single_step_SPA_original}. 
Moreover,    \cite[Theorem~1]{jin2024improved} claims that it can replace $\sigma_r(W)$ in Theorem~\ref{th:spanoise} by $\sigma_{r-1}(W)$ in the bound on the noise and in the error bound. However, this is not possible, as exemplified by the following counter example.  Suppose in fact that $W = \begin{pmatrix}
        0 & 0\\1&0
    \end{pmatrix}$, $r=2$ and $X = \begin{pmatrix}
        0& \ve & 0 \\ 1& 1 & 0
    \end{pmatrix}$ for a small enough $\ve \ge 0$. Notice that $X = WH + N$ for a separable and column stochastic $H = \begin{pmatrix}
        1 & 1 & 0 \\ 0 &0 & 1 
    \end{pmatrix}$ and with $K(N)= \ve$.  Then SPA identifies the first two columns of $X$ as the vertices of the simplex, and the error is at least $1$ for any $\ve$, that does not coincide with the estimation $\mathcal O(\ve K(W)^2/\sigma_{r-1}(W)^2 ) = \mathcal O(\ve )$. 

Before proving Theorem~\ref{theo:single_step_SPA}, let us state the following two lemmas. 
\begin{lemma}
\label{lem:median_point}      Let $W\in \f R^{m\times r}$ with $r\le m$. If $\varepsilon\le \sigma_r(W)\mathcal K(W)^{-1}/8$, then for every couple of distinct indices $i,j$   \[
    \left\| \frac{w_i+w_j}2 \right\|\le \max\{\|w_i\|,\|w_j\|\} -2\varepsilon. 
    \] 
\end{lemma}
\begin{proof}
  See Appendix~\ref{app:proofsLem1}. 
\end{proof}

\begin{lemma}
\label{lem:decomposition_convex}    Let $V$ be a convex polytope with vertices $\{v_0,v_1,\dots,v_r\}$. Let $C$ be a closed convex set containing $v_1,\dots,v_r$ but not $v_0$. Let $w_i = v_0 + \alpha_i(v_i-v_0)$ where $0<\alpha_i\le 1$ is the smallest value for which $w_i\in C$. Then the set $V\setminus C$ is contained in $\conv(v_0,w_1,\dots,w_r)$. 
\end{lemma}
\begin{proof}
See Appendix~\ref{app:proofsLem2}. 
\end{proof}

We can now prove Theorem~\ref{theo:single_step_SPA}. 
{
}

\begin{proof}[Proof of Theorem~\ref{theo:single_step_SPA}]
The vector $x$ is equal to $w + s$ where $w=[W\,\, Q]h$ and  $h,s$ are columns of $H$ and $N$ respectively, so $\|x\| < \|w\| +\varepsilon$. At the same time, given $\|w_i\| = K(W)$ the longest column of $W$, there exists a column of $X$ that is equal to $w_i$ up to a perturbation of norm at most $\varepsilon$, due to the separability of the matrix $H$, so $\|x\|> K(W) -\varepsilon$. It implies that there exists a vector $w$ in the convex hull of $[W\,\, Q]$ for which 
\[
w\in \mc S := \{ y\in \conv([W\,\, Q]) \,\,|\,\, \|y\| > K(W)-2\varepsilon \}. 
\]
Since
\[
K(W) - 4\ve  \ge K(W) - \sigma_k(W)\mathcal K(W)^{-1}/2\ge  K(W)/2 \ge K(Q), 
\]
then all vectors in $\conv(Q)$ have norm at most $K(W) -4\varepsilon$, and $\|w_i + q_j\|/2\le K(W)-2\varepsilon$.
Call $M := [W\,\, Q]$ so that $m_i = w_i$ for $i\le k$ and $m_i = q_{i-k}$ for $i>k$.  
By the hypothesis on $\varepsilon$ and Lemma \ref{lem:median_point}, 
all the vectors in the polytope $\{y=[W\,\, Q]h \,\, | \,\, e^\top h = 1, \, h\ge 0, \|h\|_{\infty}\le 1/2 \}$ have norm at most $K(W) -2\varepsilon$, since its vertices are the median points between different columns of $[W\,\, Q]$. As a consequence,
\[
\mc S = \bigcup_{i} \Big( \mc S\cap \{ y=[W\,\, Q]h \,\, | \,\, e^\top h = 1, \, h\ge 0, h_i\ge  1/2  \}\Big) := \bigcup_{i} \Big( \mc S \cap \mc V_i \Big) := \cup_{i\le k} \mc S_i
\]
where the union is disjoint. Notice that the $\mc S_i$ associated with $\|m_i \|\le K(W) - 2\varepsilon$  are empty since the vertices of $\mc V_i$ are the median points between $m_i$  and the columns of $M$ (itself included). In particular $\mc S_i = \emptyset$ for $i> k$ since  $K(Q) \le K(W) - 4\varepsilon$.
Let now  $\|w_i\|>K(W) - 2\varepsilon$.  Since the ball with radius $K(W)-2\varepsilon$ is convex, closed and contains $(w_i+m_j)/2$ for all $j\ne i$, but not $w_i$,  we can find points $\wt m_{ij}$ of norm exactly $K(W) - 2\varepsilon$ on the segments connecting $w_i$ with $(w_i+m_j)/2$. By convention, denote $\wt m_{ii}:= w_i$. Since $\mc S_i = \mc S \cap \mc V_i = (\conv(M)\setminus B(K(W)-2\varepsilon) )\cap \mc V_i = \mc V_i\setminus B(K(W)-2\varepsilon)$, we can use Lemma \ref{lem:decomposition_convex} and find that $\mc S_i\cu \conv\{\wt m_{i1},\dots,\wt m_{ir}\}$.  The point $w$ is in one of the non empty $\mc S_i$, so
\begin{equation}\label{eq:max_err}
    \min_i \| x - w_i\|= \min_i \|w + s - w_i\|\le \varepsilon + \max_{i:\mc S_i\ne \emptyset}\max_{y\in \mc S_i}\|y -w_i\|
\le \varepsilon + \max_{i:\mc S_i\ne \emptyset}\max_{j\ne i}\|\wt m_{ij} -w_i\|.
\end{equation}
Let $\delta_{ij}\in (0,1/2]$ be the value such that $\wt m_{ij} = w_i + \delta_{ij}(m_j-w_i)$ for $j\ne i$, and notice that $\|\wt m_{ij}-w_i\| = \delta_{ij}\|m_j-w_i\|$. From $\|\wt m_{ij}\| = K(W) - 2\varepsilon$, we obtain the following  equation in $\delta_{ij}$ 
\begin{equation}
    \label{eq:deltaij}(K(W)-2\varepsilon)^2 =    \|w_i\|^2 + \delta_{ij}^2 \|m_j-w_i\|^2 +2\delta_{ij} w_i^\top (m_j-w_i).
\end{equation}
If $\|m_j\|\le K(W)-2\varepsilon$, then $\delta_{ij}$ will be the only solution between $[0,1]$. Instead, if  $\|m_j\|> K(W)-2\varepsilon$, then the two solutions of the above equation will be $\delta_{ij}$ and $1-\delta_{ji}$, where $\delta_{ij}\le \frac 12 \le 1-\delta_{ji}$. In this case, if we further assume $\|w_i\| \ge \|m_j\|$, we get 
\[
\delta_{ij} + 1 -\delta_{ji} = \frac{2w_i^\top (w_i-m_j)}{\|m_j-w_i\|^2} = \frac{2\|w_i\|^2 - 2w_i^\top m_j}{\|m_j-w_i\|^2} \ge \frac{\|w_i\|^2+\|m_j\|^2 - 2w_i^\top m_j}{\|m_j-w_i\|^2} = 1,
\]
so $\delta_{ij}\ge \delta_{ji}$ and
\[
\|\wt m_{ij}-w_i\| = \delta_{ij}\|m_j-w_i\| \ge \delta_{ji}\|m_j-w_i\| = \|\wt m_{ji}-m_j\|.
\]
Thus, to estimate \eqref{eq:max_err}, we just need to consider the indices $i,j$ such that $\|w_i\| \ge \max\{\|m_j\|,  K(W)-2\varepsilon\}$.  For such couples of indices, $\delta_{ij}$ will always be the smallest solution of \eqref{eq:deltaij}, 
and $w_i^\top (w_i-m_j) \ge \|m_j-w_i\|^2/2>0$. From the inequality $a-\sqrt{a^2-x}\le x/a$ that holds for any $0\le x\le a^2$ and $0<a $, 
\begin{align*} \delta_{ij} &= \frac{
w_i^\top (w_i-m_j) - \sqrt{
(w_i^\top (w_i-m_j))^2 - \|m_j-w_i\|^2 (\|w_i\|^2 - (K(W)-2\varepsilon)^2)
} 
}{ \|m_j-w_i\|^2}\le \frac{
  \|w_i\|^2 - (K(W)-2\varepsilon)^2
}{  w_i^\top (w_i-m_j)}
\\&\le  
\frac{
2(  \|w_i\| - (K(W)-2\varepsilon))
(  \|w_i\|+ (K(W)-2\varepsilon))
}{ \|m_j-w_i\|^2}
\le
\frac{
 8 \|w_i\|\varepsilon
}{  \|m_j-w_i\|^2}
\le
\frac{
 8 K(W)\varepsilon
}{\|m_j-w_i\|^2}. 
\end{align*}
This allows us conclude that 
\begin{equation}
    \label{eq:local_error}
\|\wt m_{ij}-w_i\| = \delta_{ij}\|m_j-w_i\|
\le
\frac{
 8 K(W)\varepsilon
}{\|m_j-w_i\|} \le 
\frac{
 32 K(W)\varepsilon
}{\sigma_k(W)}, 
\end{equation}
because if $m_j = w_j$ then $\|m_j-w_i\| \ge \sqrt 2\sigma_k(W)\ge \sigma_k(W)/4$  and if $m_j = q_s$ then 
\[
\|m_j-w_i\| \ge  \|w_i\| - \|m_j\| \ge K(W) - 2\varepsilon -K(W)/2
\ge K(W)/2  - \sigma_k(W)\mathcal K(W)^{-1}/4
\ge \sigma_k(W)/4. 
\]
Finally,  by \eqref{eq:max_err},
\[
 \min_i \| x - w_i\|\le  \varepsilon + \max_{i:\mc S_i\ne \emptyset}\max_{j\ne i}\|\wt m_{ij} -w_i\|\le \varepsilon +\frac{
 32 K(W)\varepsilon
}{\sigma_k(W)}\le 33\mathcal K(W)\varepsilon.
\]
\end{proof}

The consequences of Theorem~\ref{theo:single_step_SPA} are the following.  
\begin{itemize}
    \item The first step of SPA is more robust than the next ones. The reason is that the projection step might lead to an increase of the error bound by a factor of $\mathcal K(W)$, so that Theorem~\ref{th:spanoise} becomes tight after the first step of SPA. We will prove this in Section~\ref{sec:tighT-SPA} with an example.  \revise{From a practical point, this is an interesting information: we know that the first extracted vertex by SPA is more reliable than the next ones.} 

    \item When $H^\top  e = e$, one can use T-SPA. Since T-SPA only translates at the first step while increasing the conditioning by a a factor at most two (see Section~\ref{sec:draw1}), the first two steps of T-SPA follow the bound of Theorem~\ref{theo:single_step_SPA} and are therefore more robust than the next ones.


\item In the proof of Theorem~\ref{theo:single_step_SPA}, the assumption that $H$ is separable can be relaxed: we only need the columns of $W$ to appear as columns of $X$ (up to the noise), not  that of $Q$. In other words, only $H(1:k,:)$ needs to be separable. 
We will use this observation in the proof of the  Theorem~\ref{theo:rank2_SPA_Q} in the next section.

    \item Since the only hypothesis on $Q$ is a bound on the norm of its columns, Theorem \ref{theo:single_step_SPA} generalizes easily to substochastic matrices $H$ (that is, $H^\top e \leq e$). In fact, we only need to add a zero column to $Q$ and a row to the substochastic $H$ to extend it to a stochastic matrix. 
     
    
\end{itemize}

 \subsection{SPA when $r=2$} \label{sec:SPAr2} 

Another less direct consequence of Theorem~\ref{theo:single_step_SPA} is that 
we can solve rank-2 separable NMF with better guarantees, with a bound on the noise of $\varepsilon \leq \mathcal{O}\left(\frac{\sigma_r(W)}{ \mathcal K(W)} \right)$, and 
an error bound in $\mathcal{O}\left(\varepsilon  \mathcal K(W) \right)$. 

Unfortunately, we show in the next section that the following steps of SPA when $r\ge 3$ cannot be as good and that the bound of Theorem~\ref{th:spanoise} is asymptotically tight.  


{

}

\begin{theorem}
    \label{theo:rank2_SPA_Q}    Let $W\in \f R^{m\times 2}$, $Q\in \f R^{m\times (r-2)}$ and $X = [W\,\, Q]H + N\in\f R^{m\times n}$ with $2\le m$, where  $W$ is full rank, 
    $K(N)< \varepsilon\le \ma O(\sigma_2(W)\mathcal K(W)^{-1})$, 
    $K(Q) \le \sigma_2(W)/2$ 
    and $H\in \f R_+^{r\times n}$ has stochastic columns and two of its columns are $e_1, e_2$.
Suppose $x_1, x_2$ are the two columns chosen in order by the SPA algorithm applied on $X$. Then we have that  
\[
\min_{\pi } \max_{i=1,2} \| x_i - w_{\pi(i)}\|  \le \ma O(\mathcal K(W)\varepsilon), 
\]
where $\pi$ is either the identity function or the transposition $\pi(i)= 3-i$.
\end{theorem}
\begin{proof}
   Suppose $\ve \le \sigma_2(W)\mathcal K(W)^{-1}/305$. Since $1/305\le 1/8$,
    we can apply Theorem \ref{theo:single_step_SPA} and find that, up to a permutation of the columns of $W$ and $H$, SPA chooses at the first step the column $x_1$ and $\|r_1\| \le 33\mathcal K(W)\ve$ where $r_1:= x_1-w_1$. Notice that since $x_1$ is the column of $X$ with the largest norm and $H$  has among its columns the vectors $e_1, e_2$, then $w_i + n_i$ are columns of $X$ and in particular $$\|x_1\| \ge \max_i \|w_i+n_i\| \ge \max_i \|w_i\| - \varepsilon = K(W) -\ve.$$
Let $P_v:=I-vv^\top /\|v\|^2$ be the orthogonal projection to the hyperplane orthogonal to the vector $v$, and let $c = 1/305$. 

\begin{align*}
    \|P_{x_1}-P_{w_1}\|  &= \left\| \frac{x_1x_1^\top }{\|x_1\|^2} -  \frac{w_1w_1^\top }{\|w_1\|^2}\right\| = 
    \left\| \frac{(r_1+w_1)(r_1+w_1)^\top }{\|x_1\|^2} -  \frac{w_1w_1^\top }{\|w_1\|^2}\right\| \\
    & = 
    \left\| \frac{r_1r_1^\top  + r_1w_1^\top  + w_1r_1^\top }{\|x_1\|^2} + w_1w_1^\top \left( \frac{1}{\|x_1\|^2}-  \frac{1}{\|w_1\|^2}\right)\right\|\\
     & \le 
 \frac{\|r_1\|^2 + 2\|r_1\|\|w_1\| }{\|x_1\|^2} +  \frac{|\|w_1\|^2- \|x_1\|^2|}{\|x_1\|^2}
  \\    & \le 
 \frac{\|r_1\|^2 + 2\|r_1\|\|w_1\| }{\|x_1\|^2} +  \frac{2\|w_1\| \|r_1\| + \|r_1\|^2}{\|x_1\|^2}
 = 2 \|r\|\frac{\|r_1\|+ 2\|w_1\| }{\|x_1\|^2}\\
 & \le 2 \|r_1\|\frac{33\mathcal K(W)\ve+ 2K(W)}{(K(W) - \ve)^2}
 \le 2 \|r_1\|\frac{33c\sigma_2(W)+ 2K(W)}{(K(W) - c\sigma_2(W))^2}\\
 &\le \frac {2 \|r_1\|}{K(W)}\frac{33c+ 2}{(1 - c)^2} \le  \frac { \mathcal K(W)\ve}{K(W)}\gamma(c), 
    \end{align*}
    where $\gamma(c) = 66\frac{33c+ 2}{(1 - c)^2}\le 141$.  
    As a consequence,    \begin{align*}
        \|P_{x_1}w_2\| &\ge  \|P_{w_1}w_2\| - \|(P_{x_1}-P_{w_1})w_2\|
        \ge \sigma_2(W) - \|P_{x_1}-P_{w_1}\|\|w_2\| \\
        &\ge \sigma_2(W) - \frac { \mathcal K(W)\ve}{K(W)}\gamma(c)\|w_2\| 
        \ge (1-c\gamma(c))\sigma_2(W) 
        \ge\frac {1-c\gamma(c)}c\mathcal K(W)\ve .
    \end{align*}
If we now let $\wt X := P_{x_1}X = P_{x_1}[W\,\, Q]H + P_{x_1}N =: [\wt W\,\, \wt Q]H + \wt N$, then the second step SPA identifies $\wt x_2,$ the column of  $\wt X$ with the largest norm.  
Notice that $\wt w_2 =P_{x_1}w_2 $ and
\begin{align*}
    \|\wt w_1\| = \|P_{x_1}w_1&\| \le \|x_1-w_1\|\le 33\mathcal K(W)\ve \le \frac{33c}{ 1- c\gamma(c)}\|\wt w_2\|\le  \frac 12\frac 1{1-c\gamma(c)} \|\wt w_2\|<\|\wt w_2\|.
\end{align*}
Notice moreover that  $K(\wt N)\le K(N)\le \ve$ and
$$K(\wt Q) \le K(Q) \le \frac{ \sigma_2(W)}2 \le 
 \frac 12\frac 1{1-c\gamma(c)} \|\wt w_2\|<\|\wt w_2\|.
$$
We can write $\wt x_2 = \lambda \wt w + (1-\lambda) \wt w_2 + \wt n_2$ for a $0\le \lambda\le 1$, for some $\wt w$ in the convex hull of $[\wt Q\,\, \wt w_1]$.  Since
$$\|\wt w\| \le      \frac 12\frac 1{1-c\gamma(c)} \|\wt w_2\|<\|\wt w_2\|$$
and since $H$ has $e_2$ as a column, then one of the columns of $\wt X$ is equal to $\wt w_2+\wt n$ for some column $\wt n$ of $\wt N$, but $\wt x_2$  is its column with largest norm, so
\begin{align*}
 \|\wt w_2 \| - \ve    &\le \|\wt w_2 + \wt n\| \le \|\wt x_2\| \le \lambda \|\wt w\| + (1-\lambda)\|\wt w_2\| + \ve
 \implies 
 \lambda (\| \wt w_2\| - \|\wt w\| ) \le 2\ve \\
 \implies 
 \lambda &\le \frac{2\ve}{\| \wt w_2\| - \|\wt w\|}
 \le \frac{2\ve}{(1 -\frac 12\frac 1{1-c\gamma(c)})\|\wt w_2\|}\revise{.}  
\end{align*}
Since the projection is linear, we have $x_2 = \lambda w + (1-\lambda)w_2 +  n_2$ for some $w$ in the convex hull of $[ Q\,\, w_1]$  and
\begin{align*}
    \|x_2-w_2\| &\le \ve + \lambda \|w-w_2\| \le
    \ve +\frac{2\ve\|w_1-w_2\|}{(1 -\frac 12\frac 1{1-c\gamma(c)})\|\wt w_2\|}\\
    &\le  
     \ve +\frac{4 K(W)}{(1 -\frac 12\frac 1{1-c\gamma(c)})(1-c\gamma(c)) \sigma_2(W)}\ve 
    \\
    &=
    \left(1+ \beta(c)\mathcal K(W)\right)\ve
    \le 
    \left(1+ \beta(c)\right)\mathcal K(W)\ve \revise{,}  
\end{align*}
where $$\beta(c) = \frac{4 }{(1 -\frac 12\frac 1{1-c\gamma(c)})(1-c\gamma(c)) }
=  \frac{4 }{\frac 12-c\gamma(c) }
\le 99.$$
\end{proof}

\begin{remark}
   \label{rem:better_constants_rank_2_SPA} 
  Theorem \ref{theo:rank2_SPA_Q} does not give an explicit relation between the constants in  the noise bound $\ve\le \ma O(\sigma_2(W)\mathcal K(W)^{-1})$ and the constants of the error bound   $\min_{\pi \in S^2} \max_{i=1,2} \| x_i - w_{\pi(i)}\|  \le \ma O(\mathcal K(W)\varepsilon)$ as in Theorem \ref{theo:single_step_SPA}. Carefully retracing the steps in the proof, we can find that if  $\varepsilon\le c\sigma_2(W)\mathcal K(W)^{-1}$ with $c\le 1/150$, then the error in the first and second step of SPA are respectively bounded by
    \[
   \left[ 1 + \frac{16 }{1 - 4c}\right ]\mathcal K(W) \ve \le 18\mathcal K(W) \ve,\qquad 
    \left[ 1 + \frac {8 (1 - 4 c)^2 (1 - c)^2}{1 - 146 c - 547 c^2 + 376 c^3 - 48 c^4}     \right]
    \mathcal K(W) \ve .
    \]
If for example we want to keep the same error  $33\mathcal K(W) \ve$ as in Theorem \ref{theo:single_step_SPA}, we need $c\le 1/196$. If instead we want to match the error bound $18\mathcal K(W) \ve$ of the first step, we would need $c\le 1/271$. \\
Note that the second formula above, for the second step error, can be simplified using the inequality \[
    \left[ 1 + \frac {8 (1 - 4 c)^2 (1 - c)^2}{1 - 146 c - 547 c^2 + 376 c^3 - 48 c^4}     \right]
    \le 9 + 4400 c
    \]
 that holds for any $c\le 1/200$. 
\end{remark}

\section{Improved bounds for the first two steps of T-SPA and when $r=3$} \label{sec:T-SPA} 

T-SPA is the same algorithm as SPA, except at the first step when it performs a translation instead of an orthogonal projection. 
Because of that, the bounds for T-SPA will be ``one step ahead'' of that of SPA. More precisely, we show improved   bounds for the first two steps of T-SPA in Section~\ref{sec:T-SPAfirsttwo} (Theorem~\ref{theo:firsttwo_T-SPA}), and when $r=3$ in Section~\ref{sec:T-SPAr3} (Theorem~\ref{theo:rank3_T-SPA}).

\subsection{First two steps of T-SPA} \label{sec:T-SPAfirsttwo}

Similarly as done for the first step of SPA in Section~\ref{sec:firststepSPA}, we prove in this section that T-SPA has bounds on the noise of 
$\mathcal{O}\left(\sigma_r(W) / \mathcal K(W) \right)$  and 
error bounds of  $\mathcal{O}\left(\varepsilon  \mathcal K(W) \right)$ for the first two steps. 

\begin{theorem}
    \label{theo:firsttwo_T-SPA}    Let $W\in \f R^{m\times r}$ and $X = WH + N\in\f R^{m\times n}$ with $r\le \min\{m,n\}$, where  $W$ is full rank, 
    $K(N)< \varepsilon\le \ma O(\sigma_r(W)\mathcal K(W)^{-1})$
    and $H\in \f R_+^{r\times n}$ is separable and has stochastic columns. 
Given $x_1, x_2$ the first two columns chosen by the T-SPA algorithm applied on $X$, we have that  
\[
\min_{\pi \in S^r} \max_{i=1,2} \| x_i -w_{\pi(i)}\|   \le \ma O(\mathcal K(W)\varepsilon),
\]
where  $S^r$ is the set of permutations of $\{1,2,\dots,r\}$.
\end{theorem}

Before proving Theorem~\ref{theo:firsttwo_T-SPA}, we need two auxiliary lemmas. 
\begin{lemma}
\label{lem:translated_least_sv} Let $W\in \f R^{m\times r}$ where $r\le m$ and $W$ is full rank. Let  $v\in \conv(W)$ be such that in the convex combination $v=\sum_i\lambda_iw_i$ we have $\lambda_r\ge 1/2$. If $\hat W$ are the first $r-1$ columns of the matrix $W-ve^\top $, then
 \[
 \sigma_r(W) \le \frac{2}{2-\sqrt 2} \sigma_{r-1}(\hat W),\qquad K(\hat W)\le 2 K(W).
 \]
 \end{lemma}
\begin{proof}
   See Appendix~\ref{app:proofsLem3}.
\end{proof}

\begin{lemma}
    \label{lem:first_step_TSPA} Let $W\in \f R^{m\times r}$ and $X = WH + N\in\f R^{m\times n}$ with $r\le \min\{m,n\}$, where  $W$ is full rank,      $K(N)< \varepsilon\le c\sigma_r(W)\mathcal K(W)^{-1}$, $c\le 1/233$,
    and $H\in \f R_+^{r\times n}$ is separable and has stochastic columns. 
Suppose $x_1$ is the first column chosen by the T-SPA when applied on $X$, and it is associated to $w_1$ in the sense of Theorem \ref{theo:single_step_SPA}, i.e. $x_1 = Wh_1 + n_1$ where $h_{1,1}\ge 1/2$. Call $w:= Wh_1$ and let 
\[
     \hat X := X - x_1e^\top  = (W-we^\top )H + (N-n_1e^\top) := [r_1 \,\,\hat W] H + \wt N, 
     \]
where $r_1 := w_1-w$. Then $\hat W$ is full rank and
\begin{itemize}
    \item $\mathcal K(\hat W)\le 
\frac { 4}{2-\sqrt 2  }\mathcal K( W),$
    \item $ \sigma_r(W) \mathcal K(W)^{-1} \le 
\frac { 4}{3-2\sqrt 2}  \sigma_{r-1}(\hat W) \mathcal  K(\hat W)^{-1},$
 \item $ \|r_1\| \le \frac 12 \sigma_{r-1}(\hat W)$.
\end{itemize}
\end{lemma}
\begin{proof}
   See Appendix~\ref{app:proofsLem4}. 
\end{proof}

We can now prove Theorem~\ref{theo:firsttwo_T-SPA}. 

\begin{proof}[Proof of Theorem~\ref{theo:firsttwo_T-SPA}]   
Up to a permutation of the columns of $X$, $W$ and $H$, T-SPA chooses at the first step the column $x_1$ 
and it is associated to $w_1$ in the sense of Theorem \ref{theo:single_step_SPA}, i.e. $x_1 = Wh_1 + n_1$ where $h_{1,1}\ge 1/2$. Call $w:= Wh_1$ and let 
\[
     \hat X := X - x_1e^\top  = (W-we^\top )H + (N-n_1e^\top) := [r_1 \,\,\hat W] H + \wt N, 
     \]
where $r_1 := w_1-w$ and $K(\wt N) < 2\ve$. 
Suppose $\ve \le c\sigma_r(W)\mathcal K(W)^{-1}$ with $c = 1/374$.
Lemma \ref{lem:first_step_TSPA}  tells us that  $\hat W$ is full rank, $ \|r_1\| \le \frac 12 \sigma_{r-1}(\hat W)\le \frac 12 K(\hat W)$ and
  \[
 \ve \le 
 c \sigma_r(W) \mathcal K(W)^{-1}\le 
 \frac { 4c}{3-2\sqrt 2} \sigma_{r-1}(\hat W)\mathcal   K(\hat W)^{-1}
 \le \sigma_{r-1}(\hat W)\mathcal   K(\hat W)^{-1}/16 \revise{.} 
 \]
We can thus use again Theorem \ref{theo:single_step_SPA} to obtain that, up to permutations, the second step of T-SPA extracts the column $\hat x_2 = x_2-x_1$ of $\hat X$ and it is close to the first column of $\hat W$, that is $\hat w_1 = w_2-w$,  in the sense $\|\hat x_2 -\hat w_1\|\le 66\mathcal K(\hat W)\varepsilon$. This lets us conclude that $x_2$ is close to $w_2$ as, by Lemma \ref{lem:first_step_TSPA},       \[
      \|x_2 - w_2\|  \le  
\|\hat x_2 -\hat w_1\|  + \|x_1-w \|  \le 67\mathcal K(\hat W)\varepsilon\le 
67 \frac {4}{2-\sqrt 2}    \mathcal K( W)\varepsilon.
\]
\end{proof}

\subsection{T-SPA when $r=3$} \label{sec:T-SPAr3}

As done for SPA when $r=2$ in Section~\ref{sec:SPAr2}, we prove in this section that 
 when $r=3$, T-SPA has 
 a bound on the noise of $\varepsilon \leq \mathcal{O}\left(\frac{\sigma_r(W)}{ \mathcal K(W)} \right)$, and 
an error bound in $\mathcal{O}\left(\varepsilon  \mathcal K(W) \right)$. 

\begin{theorem}
    \label{theo:rank3_T-SPA}    Let $W\in \f R^{m\times 3}$ and $X = WH + N\in\f R^{m\times n}$ with $3\le \min\{m,n\}$, where  $W$ is full rank, 
     $K(N)< \varepsilon\le \ma O(\sigma_3(W)\mathcal K(W)^{-1})$
    and $H\in \f R_+^{3\times n}$ is separable and has stochastic columns. 
Given $x_1, x_2,x_3$ the three columns chosen by the T-SPA algorithm applied on $X$, we have that  
\[
\min_{\pi \in S^3}\max_{i=1,2,3}\| x_i -w_{\pi(i)}\|   \le \ma O(\mathcal K(W)\varepsilon).
\]
\end{theorem}
\begin{proof}
Up to a permutation of the columns of $X$, $W$ and $H$, T-SPA chooses at the first step the column $x_1$ 
and it is associated to $w_1$ in the sense of Theorem \ref{theo:single_step_SPA}, i.e. $x_1 = Wh_1 + n_1$ where $h_{1,1}\ge 1/2$. Call $w:= Wh_1$ and let 
\[
     \hat X := X - x_1e^\top  = (W-we^\top )H + (N-n_1e^\top) := [r_1 \,\,\hat W] H + \wt N, 
     \]
where $r_1 := w_1-w$ and $K(\wt N) < 2\ve$. 
  Suppose $\varepsilon\le c\sigma_3(W)\mathcal K(W)^{-1}$ with $c\le 1/233$, sot that we can apply Lemma \ref{lem:first_step_TSPA} that tells us that  $\hat W$ is full rank, $ \|r_1\| \le \frac 12 \sigma_{2}(\hat W)$ and
  \[
  \ve \le 
 c \sigma_3(W) \mathcal K(W)^{-1}\le 
 \frac { 4c}{3-2\sqrt 2} \sigma_{2}(\hat W)\mathcal   K(\hat W)^{-1}
 \le \ma O(\sigma_{2}(\hat W)\mathcal   K(\hat W)^{-1}).
 \]
We can thus use Theorem \ref{theo:rank2_SPA_Q} to obtain that, up to permutations, the remaining steps of T-SPA extract the columns $\hat x_2, \hat x_3$ of $\hat X$ and they are close to the columns of $\hat W$, that are $\hat w_1 = w_2-w$ and $\hat w_2 = w_3-w$, in the sense $\|\hat x_2 -\hat w_1\|\le \ma O(\mathcal K(\hat W)\varepsilon)$ and $\|\hat x_3 -\hat w_2\|\le \ma O(\mathcal K(\hat W)\varepsilon)$. We thus conclude that      \[
      \max_{i = 2,3}\|x_i - w_i\|  \le 
\max_{i = 2,3}\|\hat x_i -\hat w_{i-1}\| + \|x_1-w\|   \le \ma O(\mathcal K(\hat W)\varepsilon) +\ve =  \ma O(\mathcal K(\hat W)\varepsilon). 
\]
\end{proof} 
In a sense, Theorem \ref{theo:rank2_SPA_Q} and Theorem \ref{theo:rank3_T-SPA} say the same thing: Theorem \ref{theo:rank2_SPA_Q} shows that if we have some points that are approximately convex combinations of two vertices and the origin, then we are able to extract two vertices;  Theorem \ref{theo:rank3_T-SPA} says that if we have points  that are approximately convex combinations of three vertices, then once we bring the first vertex to zero, we can extract the other two.

\section{Tightness of the error bound in Theorem~\ref{th:spanoise} when $r \geq 3$}  \label{sec:tighT-SPA} 


  In this section, we prove that the error bound in Theorem~\ref{th:spanoise} is tight when $r \geq 3$ (up to the hidden constant factors), which was an important open question regarding SPA~\cite[p.~231]{gillis2020nonnegative}.    

\begin{theorem} \label{theo:tightSPA} 
For $r \geq 3$, the second step of SPA might have an error in $\revise{\Omega}\left(\varepsilon  \mathcal K^2(W) \right)$. \revise{In particular, there exist a full rank matrix $W\in \f R^{3\times 3}$ and an absolute constant $c>0$ such that, for any $\ve$ in the interval $(0,c)$, there exist  $H_\ve\in \f R^{3\times 4}$ separable and column stochastic,  and $N_\ve\in \f R^{3\times 4}$ with $K(N_\ve)\le \ve$, for which the error of the second step of SPA applied to $WH_\ve+ N_\ve$ is larger than $c\ve \mathcal K^2(W)$}. This implies that the error bound in Theorem~\ref{th:spanoise} is tight for $r \geq 3$, up to constant terms.  
\end{theorem}
\begin{proof} 
Let us provide a family of examples where SPA achieves the worst-case bound, already at the second step. 

Let $0<\delta<1/4$, $K>1$, $0<\varepsilon<K\delta^2/3$ and 
\[
W = K \begin{pmatrix}
    1 & 0 & 1/2\\
    0 & 1/2 & 1/2\\
    0 & -\delta & \delta
\end{pmatrix}.
\]
The product of the singular values of $W$ is $|\det(W)| = K^3\delta$ and $K(W) = K\le \|W\|$ since $1/4+\delta^2 < 1/2 + \delta^2 <1$.  Given the decomposition
\[
W = W_1 + W_2 := K \begin{pmatrix}
    1 & 0 & 0\\
    0 & 1/2 & 0\\
    0 & -\delta & 0
\end{pmatrix} + K \begin{pmatrix}
    0 & 0 & 1/2\\
    0 & 0& 1/2\\
    0 & 0 & \delta
\end{pmatrix}, 
\]
it is immediate to see that $\sigma_1(W_1) = K$, $\sigma_2(W_1) = K\sqrt{\delta^2 + 1/4}>K/2$ and $\|W_2\| = K\sqrt{\delta^2 + 1/2}<K$, so by triangular inequality $\sigma_2(W) \le \sigma_1(W) = \|W\| \le \|W_1\|+ \|W_2\| < 2K$ and by the interlacing property of the singular values, $\sigma_1(W) \ge \sigma_2(W) \ge \sigma_2(W_1) > K/2$. As a consequence, 
\[
 \sigma_3(W) = \frac{K^3\delta}{\sigma_1(W)\sigma_2(W)} \in \left[  \frac{K\delta}{4},4K\delta \right]\implies \frac{8}{\delta}  \ge \mathcal K(W) \ge \frac{1}{8\delta}.
\]
Notice that
\[
\frac 12 + 2\frac \ve K < \frac 12 + \frac 23 \delta^2 = \sqrt{\frac 14 + \frac 23 \delta^2 + \frac 49 \delta^4 } < \sqrt{\frac 14 + \frac 23 \delta^2 + \frac 1{36} \delta^2 }<
\sqrt{\frac 14 +  \delta^2 }
\]
so there exists $\gamma$ such that $\delta>\gamma>0$ and $\sqrt{1/4 +\delta^2} -2\varepsilon/K = \sqrt{1/4 +\gamma^2}$. Let also  
\[
\wt v:= K \begin{pmatrix}
    0 \\
     1/2 \\
    -\delta 
\end{pmatrix}
+ \frac 12\left(1-\frac \gamma\delta \right)
K \begin{pmatrix}
     1/2\\
    0\\
     2\delta
\end{pmatrix}
=
K \begin{pmatrix}
   1/4\left(1-\frac \gamma\delta \right)\\
     1/2\\
 -\gamma
\end{pmatrix},\quad
v := \wt v + \frac{\ve}{ \sqrt{\frac 14+\gamma^2}}\begin{pmatrix}
   0\\
     1/2\\
 -\gamma
\end{pmatrix}.
\]
The vector $\wt v$ is in the convex hull of $W$ since $\wt v = w_3(1-\gamma/\delta)/2 + w_2[1- (1-\gamma/\delta)/2]$, and $v$ is a $\varepsilon$-perturbation of it. 
Notice moreover that $$\|v\| \le  \ve + \|\wt v\|\le \ve + K\sqrt{\frac 1{16} + \frac 14 + \gamma^2}< K\frac {\delta^2}3 + K\sqrt{\frac 5{16} + \delta^2}<K\left(
\frac 1{3\cdot 16} + \sqrt{\frac 6{16}}
\right)<K.$$
Let 
\[
z_1 := w_2 -\frac{\ve}{ \sqrt{\frac 14+\delta^2}} \begin{pmatrix}
    0 \\
     1/2\\
     -\delta 
\end{pmatrix}=
\left( -\frac{\ve}{ \sqrt{\frac 14+\delta^2}}+K\right)
\begin{pmatrix}
    0 \\
     1/2\\
     -\delta 
\end{pmatrix}
\]
which is the second column of $W$ up to a perturbation of norm $\ve$. In particular, $\|z_1\| =\|w_2\| -\ve <K-\ve$. 
Finally, let
\[
z_2 := w_3 -\frac{\ve}{ \sqrt{\frac 14+\delta^2}} \begin{pmatrix}
    0 \\
     1/2\\
     \delta 
\end{pmatrix}
=
\left( -\frac{\ve}{ \sqrt{\frac 14+\delta^2}}+K\right)
\begin{pmatrix}
    0 \\
     1/2\\
     \delta 
\end{pmatrix}
+K
\begin{pmatrix}
     1/2\\
     0\\
    0
\end{pmatrix}
\]
which is the third column of $W$ up to a perturbation of norm  $\ve$. In particular, 
$$\|z_2\|\le \|w_3\| +\ve = K \sqrt{\frac 12 + \delta^2}      +2\ve-\ve <
 K \sqrt{\frac 12 + \frac 1{16}}  +\frac 23 K\delta^2    -\ve 
 < K \left( \frac 3{4}  +\frac 1{24} \right)    -\ve 
<K-\ve.$$ Let us build the matrix
\[
X =  \begin{pmatrix}
    K &  &  &\\
    0 & z_1 & z_2& v\\
    0 &  &  & 
\end{pmatrix}
\]
which is in the form $WH + N$ with $H$ separable and column stochastic, $W$ full rank and $K(N)\le \ve$.  

 The first step of SPA identifies the first column as the vector of largest norm and projects on its orthogonal space, obtaining the new residual matrix 
\[
\wt X = 
K \begin{pmatrix}
0 & 0 & 0& 0 \\ 
    0 & 1/2 \left(1 -\frac{\ve}{ K\sqrt{\frac 14+\delta^2}}\right) & 1/2 \left( 1-\frac{\ve}{ K\sqrt{\frac 14+\delta^2}}\right)& 1/2 \left( 1 + \frac{\ve}{ K\sqrt{\frac 14+\gamma^2}}\right)\\
    0 & - \delta \left(1 -\frac{\ve}{ K\sqrt{\frac 14+\delta^2}}\right) & \delta \left( 1-\frac{\ve}{ K\sqrt{\frac 14+\delta^2}}\right) &  - \gamma \left( 1+ \frac{\ve}{K\sqrt{\frac 14+\gamma^2}}  \right) 
\end{pmatrix}. 
\] 
 Figure~\ref{fig:sec4_3d}  represents the columns of $X$ and $\wt X$,   and Figure~\ref{fig:sec4_2d} that of $\wt X$ discarding the first coordinate which is equal to zero for all columns. 
\begin{figure}[!htb]
   \begin{minipage}{0.48\textwidth}
     \centering
     \includegraphics[height=300px]{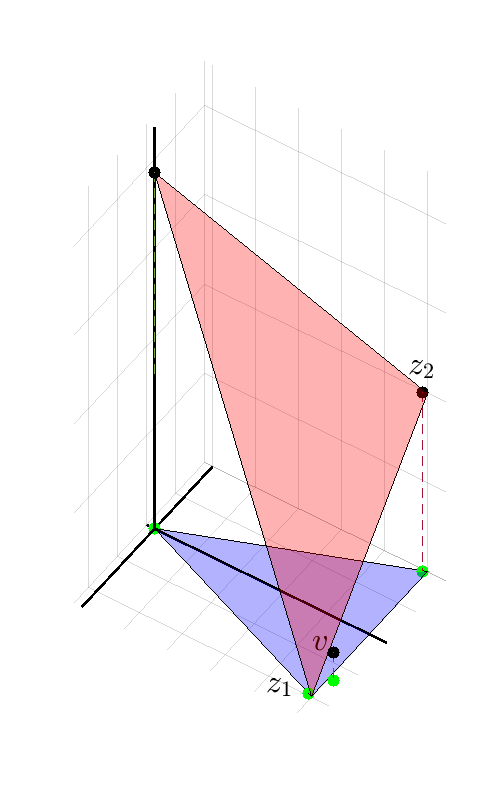}
     \caption{Columns of matrix $X$ as black dots and of matrix $\wt X$ as green dots.  \label{fig:sec4_3d}}
   \end{minipage}\hfill
   \begin{minipage}{0.48\textwidth}
     \centering
     \includegraphics[height=300px]{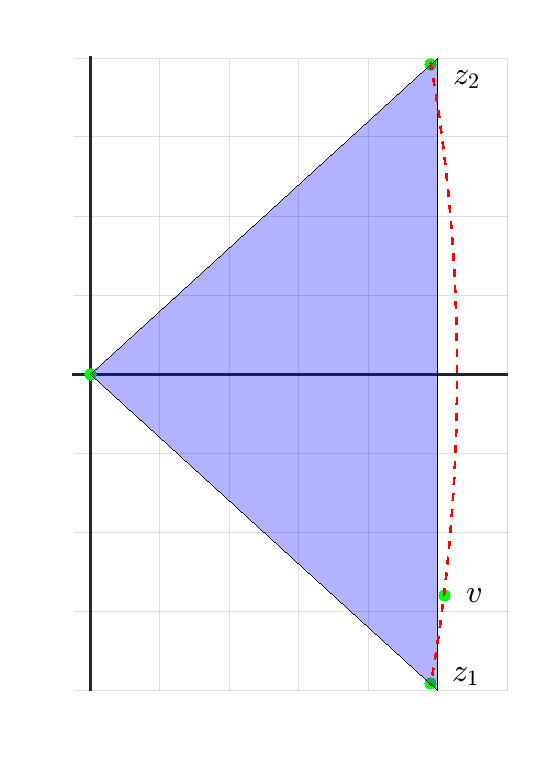}
     \caption{
     Columns of matrix $\wt X(\{2,3\},:)$. The points of the dashed red line have the same norm.
     \label{fig:sec4_2d}}
   \end{minipage}
\end{figure}

Notice that all the columns of $\wt X$ have the same norm since
\begin{align*}
    \left( \frac 14 + \gamma^2 \right)\left(\frac{\ve}{ K\sqrt{\frac 14+\gamma^2}}+1\right)^2
&= \left(\frac\ve K+\sqrt{\frac 14+\gamma^2}\right)^2\\
&= \left(-\frac\ve K+\sqrt{\frac 14+\delta^2}\right)^2 = 
 \left( \frac 14 + \delta^2 \right)\left(-\frac{\ve}{ K\sqrt{\frac 14+\delta^2}}+1\right)^2. 
\end{align*}
Hence SPA can select the last one, $v$, at the second step. 
In that case, the final error of SPA will be larger than $\min_i\|w_i-v\|$. Moreover, recall that $v$ is $\wt v$ plus a perturbation of norm $\ve$, so $\|w_i-v\| \ge \|w_i-\wt v\|-\ve$. First of all, $\|w_1-\wt v\| \ge 3K/4> 9\ve /(4 \delta^2)$ and  $\|w_3-\wt v\| \ge K/4> 3\ve /(4 \delta^2)$. For the second one, we get
\begin{align*}
    \|w_2-\wt v\| 
    &\ge K\frac {\delta -\gamma}{4\delta}
    = K\frac {\delta^2 -\gamma^2}{4\delta(\delta +\gamma)}
    \ge  K\frac {\left(\frac 14 + \delta^2\right) -\left(\frac 14 + \gamma^2\right)}{8\delta^2}\\&= 
     K\frac { \sqrt{\frac 14 + \delta^2} - \sqrt{\frac 14 + \gamma^2}}{8\delta^2}\left( \sqrt{\frac 14 + \delta^2} + \sqrt{\frac 14 + \gamma^2}\right) \ge 
      K\frac { 2\ve /K}{8\delta^2} = \frac {\ve }{4\delta^2}.
\end{align*}
As a consequence, the error will be larger than 
$$
\frac {\ve }{4\delta^2}-\ve\ge \frac {3\ve }{16\delta^2}\ge \frac {3\ve }{16} \frac{\mathcal K^2(W)}{64}.
$$ 
\end{proof}

\section{Tightness of preconditioned SPA variants} \label{sec:precond} 

In this section, we show that the two preconditioned variant of SPA described in Section~\ref{sec:draw2}, namely SPA$^2$ and MVE-SPA, have tight error bounds. 
We do this again with examples achieving the worst-case bounds.

\subsection{Tightness of SPA$^2$}  \label{sec:spaspa} 

Recall the SPA$^2$ first runs SPA on $X$ to obtain $W = X(:,\mathcal{J})$, 
and then re-runs SPA on $W^\dagger X$ to obtain $W' = X(:,\mathcal{J}')$. It turns out this strategy reduces the error bound of SPA to $\mathcal{O}\left(\varepsilon  \mathcal K(W) \right)$. 
\begin{theorem}\cite[Theorem~4.1]{gillis2015enhancing} \label{th:spaspa}
    Under the same assumptions and bounds on the noise as Theorem~\ref{th:spanoise}, the index set $\mathcal{J}'$ extracted by SPA$^2$ satisfies the following error bound: 
    \begin{equation*}
\min_{\pi \in S^r}\max_{1 \leq k \leq r} 
\|W(:,k) - X(:,\mathcal{J}'_{\pi(k)}) \| 
\leq 
\mathcal{O}\left(  \varepsilon \mathcal K(W) \right). 
\end{equation*}
\end{theorem}

Let us prove that this bound is tight, as soon as $r \geq 2$. 
\begin{theorem} \label{theo:tightSPA2}
    The error bound of Theorem~\ref{th:spaspa} is tight for $r \geq 2$.  
\end{theorem}
\begin{proof}
    Let $M$ be any positive real number and let $0<\delta < 1/2$,  $0<\ve<\delta M/6$,  $\alpha= \ve^2/(6 \delta M)<\ve /36$. Consider  
\[
W = M\begin{pmatrix}
    1 & 0 \\
    0 &  \delta
\end{pmatrix}, \qquad 
X = \begin{pmatrix}
    M & 0 & \frac{2\ve +\alpha}{\delta} \\
    0 &  \delta M -\ve &  \delta M -\ve -\alpha
\end{pmatrix} = W \begin{pmatrix}
    1 & 0 & z\\
    0 &  1 & 1-z
\end{pmatrix} + 
\begin{pmatrix}
    0 & 0 & 0 \\
    0 &  -\ve & \ve
\end{pmatrix}, 
\]
where 
\[
0<z = \frac{2\ve +\alpha}{\delta M} <  \frac{3\ve }{\delta M} <\frac 12.
\]
The columns of $X$ and $W$ are represented on Figure \ref{fig:sec6_1}. Since $0<\delta M -\ve<M$,  the squared norm of the last column of $X$ is 
\[
\left(\frac{2\ve +\alpha}{\delta}\right)^2 + \left(\delta M -\ve -\alpha\right)^2 < \left(\frac{3\ve }{\delta}\right)^2 + \left(\delta M\right)^2 < \left(\frac{M }{2}\right)^2 + \left(\frac{M }{2}\right)^2 < M^2. 
\]
 SPA thus chooses first the first index, and then the second one since $\delta M -\ve > \delta M -\ve -\alpha$. We  get the preconditioner
\[
\wt W = \begin{pmatrix}
    M & 0 \\
    0 &  \delta M -\ve
\end{pmatrix} \implies 
\wt X = \wt W^{-1} X = \begin{pmatrix}
    1 & 0 & \frac{2\ve +\alpha}{\delta M} \\
    0 &  1 &  1  -\frac{\alpha}{\delta M -\ve }
\end{pmatrix}.
\]
The columns of $\wt X$ are represented on Figure \ref{fig:sec6_2}. Since $\delta M-\ve> 5\delta M/6$, then the norm squared of the last column of $\wt X$ is now
\[
\left(\frac{2\ve +\alpha}{\delta M }\right)^2 + \left( 1  -\frac{\alpha}{\delta M -\ve }\right)^2 >
\frac{4\ve^2 }{(\delta M)^2 } + 1  -\frac{2\alpha}{\delta M -\ve }
>
\frac{4\ve^2 }{(\delta M)^2 } + 1  -\frac 1 5 \frac{2\ve^2}{(\delta M)^2  } > 1, 
\]
so the first index chosen by the algorithm SPA$^2$ is the third one. This means that the final error of SPA$^2$ will be at least the minimum distance between the last column of $X$ and the columns of $W$:
\begin{align*}
    \text{(error of SPA$^2$)}^2 &\ge \min \left\{   \left(\frac{2\ve +\alpha}{\delta } - M\right)^2 + \left( \delta M -\ve -\alpha\right)^2,
    \left(\frac{2\ve +\alpha}{\delta  }\right)^2 + \alpha^2\right\}\\&\ge 
     \min \left\{  
     \left(M- \frac{2\ve +\alpha}{\delta }\right)^2 ,
    \left(\frac{2\ve }{\delta  }\right)^2 
    \right\}\ge 
     \min \left\{  
     \left(M- \frac{3\ve}{\delta }\right)^2 ,
    \left(\frac{2\ve }{\delta  }\right)^2 
    \right\}\\&\ge 
     \min \left\{  
     \left(\frac{3\ve}{\delta }\right)^2 ,
    \left(\frac{2\ve }{\delta  }\right)^2 
    \right\} =  \left(\frac{2\ve }{\delta  }\right)^2 . 
\end{align*}
Since $\ma K(W) = 1/\delta$, we obtain 
  that the error bound of SPA$^2$ must be larger than $2\ma K(W)\ve$.
\end{proof}

\begin{figure}[!htb]
   \begin{minipage}{0.48\textwidth}
     \centering
     \includegraphics[height=170px]{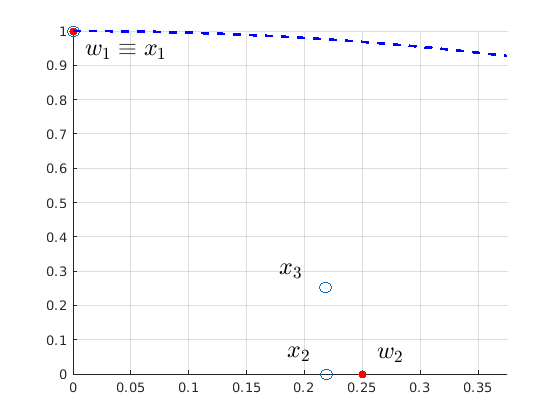}
     \caption{Columns of matrix $X$ as empty blue circles and columns of matrix $W$ as red dots.  The points on the circle (dashed blue line) have the same norm.\label{fig:sec6_1}}
   \end{minipage}\hfill
   \begin{minipage}{0.48\textwidth}
     \centering
     \includegraphics[height=170px]{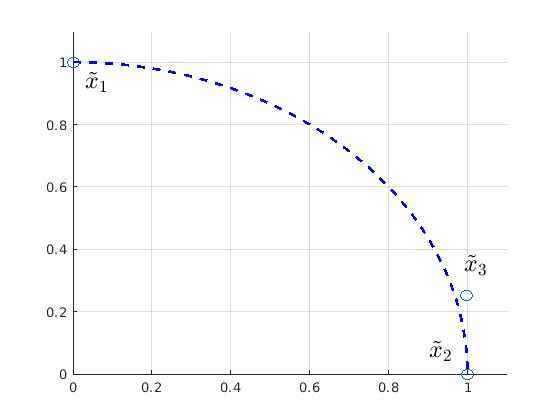}
     \caption{
     Columns of the preconditioned matrix $\wt X$. The points on the circle (dashed blue line) have the same norm. 
     \label{fig:sec6_2}}
   \end{minipage}
\end{figure}


\subsection{Tightness of MVE-SPA} \label{sec:ellipse}

MVE-SPA needs to be applied on an $r$-by-$n$ matrix, so the input matrix needs to be replaced by a low-rank approximation (e.g., using the truncated SVD) as explained in Section~\ref{sec:draw2}. 
Then MVE-SPA first identifies the minimum-volume ellipsoid centered at the origin and containing all data points, defined as $\{ v \in \mathbb{R}^r  \ | \ v^\top  A v \leq 1\}$ where $A \succ 0$, and then applies SPA on $\wt X = A^{1/2} X$ to obtain an index set $\mathcal{J}$. After the preconditioning, all data points previously on the border of the ellipsoid have exactly norm 1, whereas all the other points have norm strictly less than 1. As a consequence, it can be shown that MVE-SPA can extract any index corresponding to data points on the border of the ellipsoid, that is, $\{ v \in \mathbb{R}^r  \ | \ v^\top  A v =  1\}$~\cite{gillis2015semidefinite}; see also \cite{mizutani2014ellipsoidal}. 

\begin{theorem}~\cite[Theorem~2.9]{gillis2015semidefinite} \label{th:mvespa}
Let $X = WH + N \in \mathbb{R}^{r \times n}$ where $\rk(W) = r$, $H \in \mathbb{R}^{r \times n}_+$ is a separable matrix satisfying $H^\top  e \leq e$, and  
\begin{equation*} 
\varepsilon = K(N) 
\leq \mathcal{O}\left(  \frac{\sigma_r(W)}{r \sqrt{r}} \right). 
\end{equation*} 
Then MVE-SPA returns a set of indices $\mathcal{J}$ such that 
\begin{equation*}
\min_{\pi \in S^r}\max_{1 \leq k \leq r} \|W(:,k) - X(:,\mathcal{J}_{\pi(k)}) \|
\leq 
\mathcal{O}\left(  \varepsilon \mathcal K(W) \right). 
\end{equation*}
\end{theorem} 
\begin{theorem} \label{theo:tightMVESPA}
The error bound in Theorem~\ref{th:mvespa} is tight.  
\end{theorem} 
\begin{proof}
    We can use a similar example as for SPA$^2$ in Theorem~\ref{th:spaspa}.  
Let $M$ be any positive real number and let $0<\delta < 1$,  $0<\ve<\delta M/4$. Consider  
\[
W = M\begin{pmatrix}
    1 & 0 \\
    0 &  \delta
\end{pmatrix}, \qquad 
X = \begin{pmatrix}
    M & 0 & \frac{2\ve}{\delta} \\
    0 &  \delta M -\ve &  \delta M -\ve 
\end{pmatrix} = W \begin{pmatrix}
    1 & 0 & z\\
    0 &  1 & 1-z
\end{pmatrix} + 
\begin{pmatrix}
    0 & 0 & 0 \\
    0 &  -\ve & \ve
\end{pmatrix}, 
\]
where $z = 2\ve /(M\delta)<1/2$. The minimum volume ellipsoid going through the first two columns of $X$ is 
\[
\mathcal{E} = \{ v \ | \ 
v^\top  \diag((\delta M-\ve )^2, M^2) v 
= (\delta M-\ve )^2 v_1^2 + M^2 v_2^2 
\leq  M^2 (\delta M-\ve )^2 \}. 
\]
The third column of $X$ is outside this ellipsoid as it can be seen in Figure \ref{fig:sec6_3}, since 
\[
 \frac{4\ve^2(\delta M-\ve )^2/\delta^2 +(\delta M -\ve )^2 M^2  }{M^2 (\delta M-\ve )^2}
>1.
\]
As a consequence, the minimum volume ellipsoid containing all column\revise{s} of $X$ must have the third column on its boundary, and MVE-SPA may identify  it after the preconditioning, since all the columns of the preconditioned matrix $\wt X$ have norm $1$ as it can be seen in Figure \ref{fig:sec6_4}. 
The distance between the third column and the first is at least
\[
M -  \frac{2\ve}{\delta} >  \frac{2\ve}{\delta} =  2\ma K(W)\ve, 
\]
and the distance between the third column and the second is exactly $\frac{2\ve}{\delta} =  2\ma K(W)\ve$. This concludes that the error bound  of MVE-SPA will be
larger than $2\ma K(W)\ve$.  
\end{proof}

\begin{figure}[!htb]
   \begin{minipage}{0.48\textwidth}
     \centering
     \includegraphics[height=170px]{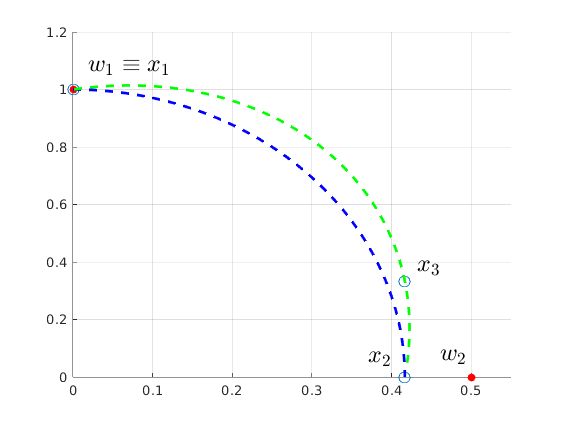}
     \caption{Columns of matrix $X$ as empty blue circles and columns of matrix $W$ as red dots.  The dashed blue line is the minimum volume ellipsoid containing $x_1$ and $x_2$. The dashed green line is the minimum volume ellipsoid containing all columns of $X$.\label{fig:sec6_3}}
   \end{minipage}\hfill
   \begin{minipage}{0.48\textwidth}
     \centering
     \includegraphics[height=170px]{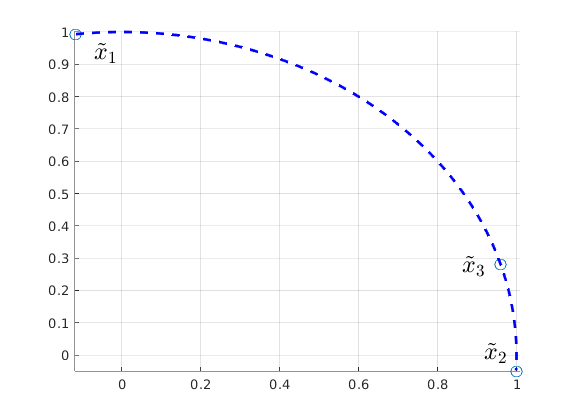}
     \caption{
     Columns of the preconditioned matrix. The points on the circle (dashed blue line) have norm $1$.
     \label{fig:sec6_4}}
   \end{minipage}
\end{figure}

\section{Translation and Lifting} \label{sec:preprocess} 

We have seen that SPA cannot extract $r$ vertices in dimension $r-1$. This can be resolved using T-SPA (see Algorithm~\ref{algo:T-SPA}), via a translation instead of a projection at the first step, given that $H^\top  e=e$, as discussed in Section~\ref{sec:draw1}. 
As briefly mentioned in Section~\ref{sec:draw1}, this is also possible by lifting the data points in one dimension higher adding a vector $c e^\top $ for $c \in \mathbb{R}$ as the last row of $X$, see~\eqref{eq:lift}; see also, e.g.,~\cite{mao2021estimating}.


In this section, we propose a new strategy, combining the lifting procedure with a translation, in order to minimize the conditioning of $W$. 
In fact, under the condition that $H^\top  e = e$, \revise{translating and shifting does not change the solution $H$ and translates and shifts $W$ in the same way} since 
\begin{equation} 
X = WH \; \iff \; \binom{X - v e^\top }{c e^\top } = \binom{W - v e^\top }{c e^\top } H, 
\end{equation} 
for any vector $v$ and constant $c$. 
The problem is as follows: find $v \in \mathbb{R}^m$ and $c \in \mathbb{R}$ such that the conditioning of $\begin{pmatrix}
	 W - v e^\top  \\
	c e^\top  
\end{pmatrix}$ 
is minimized. Of course, $W$ is unknown.


In the following, we propose a simple heuristic to tackle this problem. In order to decrease the largest singular value of $W - v e^\top $, a heuristic is to minimize the sum of its squared singular values, that is, $\|W-ve^\top\|_F^2$. The optimal solution of this problem is given by the average of the column of $W$, that is, $v = \frac{1}{r}We$. Note that, if it happens that the best rank-one approximation of $W$ has the form $u e^\top $ for some vector $u$ (in other terms, $e/\sqrt{r}$ is the first right singular vector of $W$), then the largest singular value of $W-ue^\top $ becomes $\sigma_2(W)$ which \revise{coincides with the smallest  $\sigma_1(W-ze^\top)$ among all $z$ (by the interlacing property of singular values applied to rank $1$ modifications)}.    
In practice, this average is typically well approximated by the average of the columns of $X$, given that the data points are evenly distributed in $\conv(W)$. Hence, we suggest to use $v = \frac{1}{n} Xe$. It is also possible to use more advanced strategies, e.g., take $v$ as the average of the columns of $X$ extracted by SPA~\cite{abdolali2021simplex}. However, for simplicity, we consider in this paper only  $v = \frac{1}{n} Xe$.  
Then, assuming $v = \frac{1}{r} We$, we can characterize the conditioning of $\binom{W - v e^\top }{c e^\top }$ as a function of $c$ as follows, from which we will propose a value for $c$.  
\begin{lemma} \label{lem:lemmaCvalue} 
    Suppose $W_t \in \mathbb R^{m \times r}$ with $m \geq r-1$ and the average of its columns is equal to $0$,  that is, $W_t e=0$. 
    If $W_\ell = \begin{pmatrix}
	 W_t \\
	c \; e^\top  
\end{pmatrix}$,  then \vspace{-0.3cm} 
\begin{equation} \label{eq:condWl}
        \mathcal K(W_{\ell}) = \begin{cases}
        c\sqrt r/\sigma_{r-1}(W_t) & \text{if } c\ge \|W_t \|/\sqrt r,\\
        \|W_t\|/\sigma_{r-1}(W_t) &  \text{if } \|W_t\|/\sqrt r\ge c\ge \sigma_{r-1}(W_t)/\sqrt r,\\
         \|W_t\|/(c\sqrt r) & \text{if } \sigma_{r-1}(W_t)/\sqrt r\ge c.
    \end{cases}
\end{equation} 
\end{lemma}
\begin{proof}
If we decompose $x\in \mathbb R^r$ as $x = \alpha e/\sqrt r + w\sqrt{1-\alpha^2}$ where $\|w\|^2=1$ and $w$ is orthogonal to $e$, then  
\begin{align*}
    \|W_{\ell}\|^2 &= \max_{\|x\|=1} \| W_{\ell} x\|^2  = 
\max_{\|x\|=1} \| W_t x\|^2 + (ce^\top  x)^2 = 
\max_{\|w\|=1,w^\top  e = 0, |\alpha|\le 1} (1-\alpha^2)\| W_t w\|^2 +(\alpha c)^2 r\\
&= 
\max\left\{\max_{\|w\|=1,w^\top  e = 0} \| W_t w\|^2 ,c^2 r\right\}= 
     \max\left\{ \| W_t\|^2,  c^2 r \right\}, 
\end{align*}
so $\|W_{\ell}\| =  \max\left\{ \| W_t\|,  c \sqrt{r} \right\}$. In a totally analogous way we find that $\sigma_r(W_{\ell}) = \min\left\{ \sigma_{r-1}( W_t),  c \sqrt r \right\}$.
\end{proof}



According to~\eqref{eq:condWl}, assuming $v = We/r$, \revise{and denoting by $W_t = W - ve^\top$ the $t$ranslated version of $W$, and by $W_\ell = \begin{pmatrix}
	 W_t \\
	c \; e^\top  
\end{pmatrix}$ the $\ell$ifted version of $W$,} 
the value of $c$ minimizing $\mathcal K(W_{\ell})$ is in the interval 
\[
\left[ \frac{1}{\sqrt{r}} \sigma_{r-1}\left(W-ve^\top \right),  \frac{1}{\sqrt{r}} \left\|W-ve^\top \right\| \right]. 
\] 
How can approximate this interval since $W$ is unknown?  
Since the columns of $X_t = X-ve^\top $ are within the convex hull of the columns of $W_t = W-v e^\top $, we assume $\|W_t\|_F \approx \frac{\sqrt{r}}{\sqrt{n}}\|X_t\|_F$, and hence we approximate the upper bound of the interval using the following approximation: 
$\left\|W_t \right\|_F/\sqrt{r} \approx 
\left\|X_t\right\|_F/\sqrt{n}$. 
For the lower bound, we estimate $\sigma_{r-1}(W_t)$ with $\min_j \|X_t(:,j)\|$. 
Finally, we propose to use $c$ as the average of the estimated lower and upper bounds, 
\begin{equation} \label{eq:cvalue} 
c = \frac{1}{2} \left( \min_i \|X(:,i)-v\| / \sqrt{r}  +  \left\|X - ve^\top \right\|_F/\sqrt{n} \right). 
\end{equation}
However, we have not observed a significant sensitivity of TL-SPA w.r.t.\ to the value of $c$, as hinted in Lemma~\ref{lem:lemmaCvalue}. 
For example, using the average of the absolute values of $X-ve^\top $ provides similar results.  

We will refer to the variant of SPA that applies SPA on $W_\ell$, the translated and lifted data, as translated lifted SPA (TL-SPA)\revise{; see Algorithm~\ref{algo:TL-SPA}.}

\begin{algorithm}[ht!]
\caption{\revise{Translated lifted SPA (TL-SPA)} \label{algo:TL-SPA}}
\begin{algorithmic}[1] 
\revise{
\REQUIRE The matrix $X = WH + N$ where $H \geq 0$, $H^\top  e = e$, $H$ is separable, $W$ is full column rank, 
the number~$r$ of columns to extract.  
 
\ENSURE Set of $r$ indices $\mathcal{J}$ 
such that ${X}(:,\mathcal{J}) \approx W$ (up to permutation). 

    \medskip  
    
\STATE $v = Xe / n$. \vspace{0.1cm}  

\STATE $c = \frac{1}{2} \Big( \min_i \|X(:,i)-v\| / \sqrt{r}  +  \left\|X - ve^\top \right\|_F/\sqrt{n} \Big)$.\vspace{0.1cm}  

\STATE $X_{t\ell} =  
\begin{pmatrix}
	 X-ve^\top \\
	c \; e^\top  
\end{pmatrix}$. \quad \quad \emph{\% Translate and lift}  \vspace{0.1cm}

\STATE $\mathcal J =$ SPA($X_{t\ell}$,$r$). \quad \quad \quad \emph{\% see Algorithm~\ref{algo:spa}} 
   
}
\end{algorithmic}  
\end{algorithm}

\begin{remark}[Implementation of TL-SPA] 
    As for SPA and T-SPA (see Remarks~\ref{rem:spa} and~\ref{rem:tspa}), TL-SPA can be implemented in $\mathcal{O}(r \nnz(X))$ operations (essentially, $r$ matrix-vector products). 
    In particular, $X-ve^\top$ should not be computed explicitely as $v$ could be dense. 
    See \url{https://gitlab.com/ngillis/robustSPA} for a MATLAB implementation. 
\end{remark}


\section{Numerical experiments} \label{sec:numexp}

In this section, we illustrate some of the theoretical findings by comparing the various algorithms discussed in this paper, namely: 
\begin{itemize}
    \item SPA~\cite{arora2013practical, gillis2013fast}, see Algorithm~\ref{algo:spa}. 

    \item T-SPA~\cite{arora2013practical}, see Algorithm~\ref{algo:T-SPA}.  

    \item FAW~\cite{arora2013practical}: it starts with T-SPA and then loops once over all extracted by T-SPA and tries to replace them by increasing the volume of the convex hull of the extracted columns of $X$; see also  Section~\ref{sec:draw2}. 

    \item SPA$^2$~\cite{gillis2015enhancing}:  SPA preconditioned with SPA; see  Section~\ref{sec:spaspa}.  

    \item TL-SPA: it applies SPA on the translated and lifted data matrix as proposed in Section~\ref{sec:preprocess}. 

    \item TL-SPA$^2$: it applies SPA$^2$ on the translated and lifted data matrix as described in Section~\ref{sec:preprocess}. 
    
\end{itemize}

We do not include MVE-SPA as it performs similarly as  SPA$^2$, as reported, e.g., in~\cite[Chapter 7]{gillis2020nonnegative}. 
The code used to run these experiments is available from \url{https://gitlab.com/ngillis/robustSPA}. We do not report computational times as all algorithms are extremely fast, although FAW is about $r$ times slower than SPA, T-SPA and TL-SPA, while SPA$^2$ and  TL-SPA$^2$ are about two times slower. (It takes about 10 seconds on a standard laptop to run the experiments described below, running the 6 algorithms on 6120   matrices.)

\paragraph{Synthetic data sets} 

We use the same synthetic data sets as in ~\cite[Chapter 7.4.6.1]{gillis2020nonnegative}: $X = WH+N$ where 
\begin{itemize}
    \item The entries of $W \in \mathbb{R}^{m \times r}$ are generated uniformly at random in the interval $[0,1]$, \texttt{W = rand(m,r)}. In this case, $W$ is well-conditioned with high probability.  
    To make $W$ ill-conditioned, we compute the compact SVD of $W = U \Sigma V^\top $, then update $W \leftarrow U \Sigma' V^\top $ where $\Sigma'$ is a diagonal matrix whose entries are log-spaced between $10^{-6}$ and 1, hence \revise{$\frac{\sigma_1(W)}{\sigma_r(W)} = 10^{6}$ and $\mathcal K(W)$ will be smaller, of the order of $5 \cdot 10^5$; see Table~\ref{table:cond}}. 

    \item The matrix $H = [I_r, H']\Pi \in \mathbb{R}^{r \times n}$ where $\Pi$ is a random permutation, and the columns of $H'$ either follow the Dirichlet distribution of parameters 0.5, or each column of $H'$ has two non-zero entries equal to 1/2 for all possible combinations of 2 columns of $W$, so that $n = r + \binom{r}{2}$. The latter case is the so-called middle point experiment since besides the columns of $W$, the columns of $X$ are in-between two columns of $W$. This is a more adversarial case than the Dirichlet distribution, as all data points are on edges of the boundary of $\conv(W)$.  

\item In the case of the Dirichlet distribution, each entry of $N$ follows a Gaussian distribution, \texttt{N = randn(m,n)}, and is then scaled $N \leftarrow \delta \frac{\|WH\|_F}{\|N\|_F} N$ where $\delta$ is the noise level (=ratio of the norms of $N$ and $WH$). 
For the middle point experiments, the middle points are moved towards the outside of the convex hull of $\conv(W)$: if $X(:,j)$ is a middle point, 
$N(:,j) = \delta \left( X(:,j) - \bar{w} \right)$ where $\bar{w}$ is the average of the columns of $W$. The columns of $W$ in $X$ are not perturbed. 
This is an adversarial noise: the data points are moved outside $\conv(W)$. 

\end{itemize}

Table~\ref{table:synthdata} summarizes the 4 types of data sets we will generate. 
\begin{table}[h!] 
\begin{center} 
\caption{Four types of synthetic data sets used in our 4 experiments. \label{table:synthdata}} 
\begin{tabular}{c|ccc|ccc} 
  &  $m$        &  $n$      &  $r$        &  $W$     &  $H$   &  $N$ \\ \hline 
 Exp.\ 1 & 40 & 110 & 10 
 & well-cond. & Dirichlet & Gaussian \\
  Exp.\ 2 & 40 & 110 & 10 
 & ill-cond. & Dirichlet & Gaussian \\
   Exp.\ 3 & 40 & 55 & 10 
 & well-cond. & Middle points & adversarial \\
   Exp.\ 4 & 9 & 55 & 10 
 & rank-def. & Middle points & adversarial \\
\end{tabular} 
\end{center} 
\end{table} 
For each experimental settings, we test 51 different noise levels $\delta$ and, for each noise level, we report the average accuracy (which is the proportion of the indices of the columns of $W$ correctly recovered by the corresponding algorithm) over 30 randomly generated matrices.  

Figure~\ref{fig:Dirichlet} reports the accuracy as a function of the noise level for the Dirichlet experiments (Exp.\ 1 and 2), and Figure~\ref{fig:Middlepoint} for the middle point experiments (Exp.\ 3 and 4). 

\begin{figure*}[p!]
\begin{center}
\begin{tabular}{c}
Exp.~1: Well-condition Dirichlet experiment 
\\
 \includegraphics[width=0.95\textwidth]{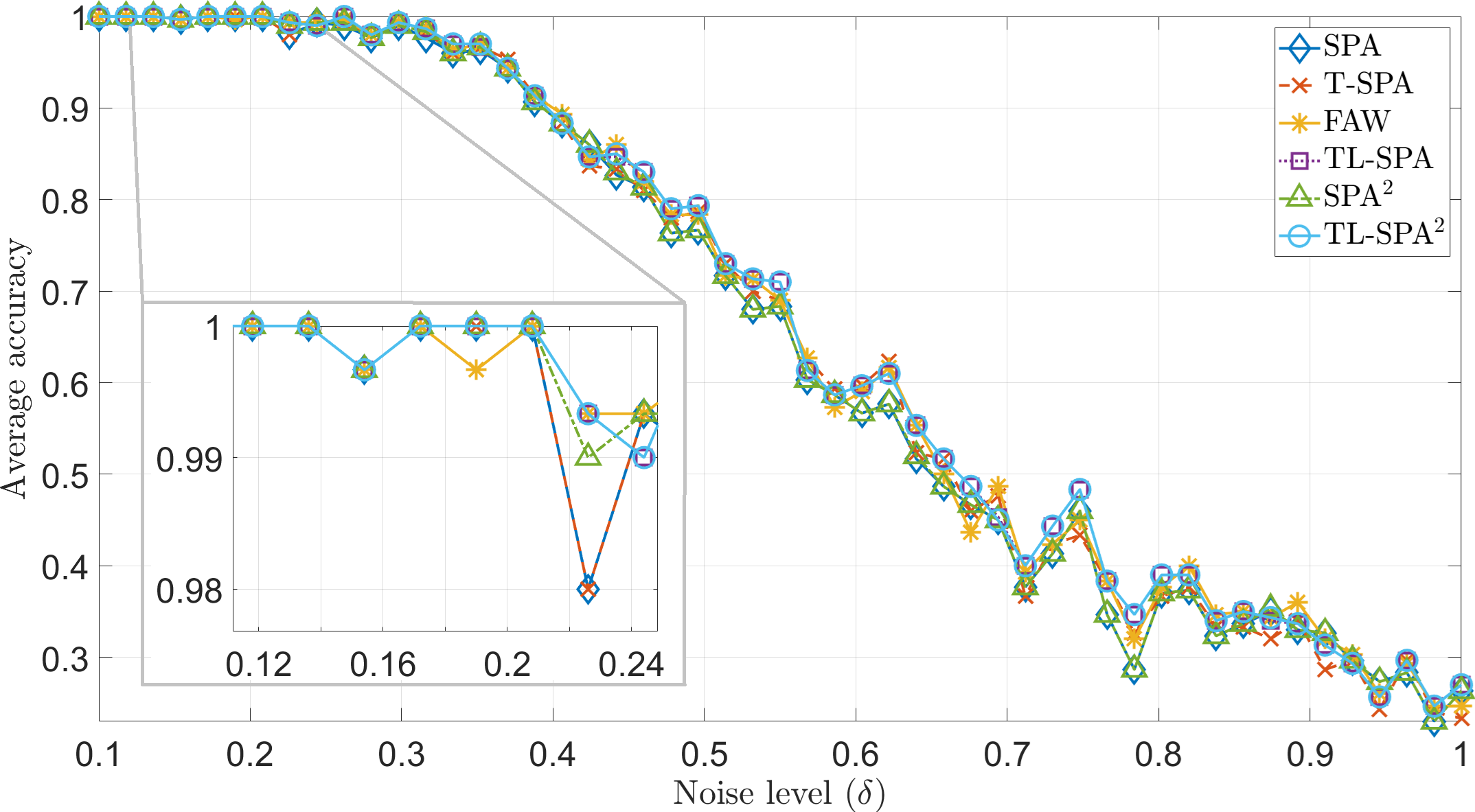} \vspace{0.2cm} \\ 
 Exp.~2:  Ill-condition Dirichlet experiment \\ 
  \includegraphics[width=0.95\textwidth]{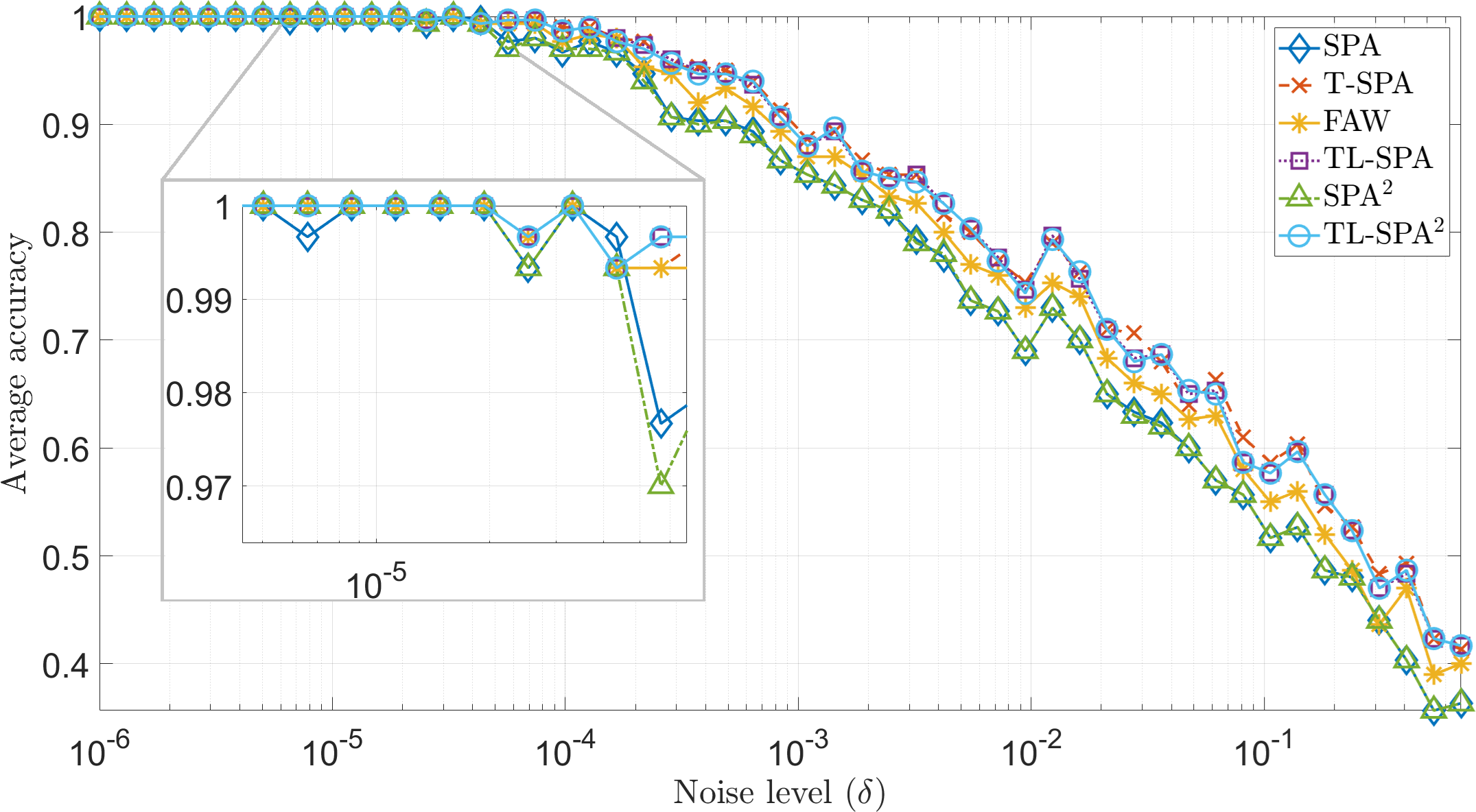} 
\end{tabular}
\caption{Accuracy vs.\ noise level for Dirichlet experiments. \label{fig:Dirichlet}}  
\end{center}
\end{figure*}

\begin{figure*}[p!]
\begin{center}
\begin{tabular}{c}
Exp.~3:  Well-condition middle point experiment \\
   \includegraphics[width=0.95\textwidth]{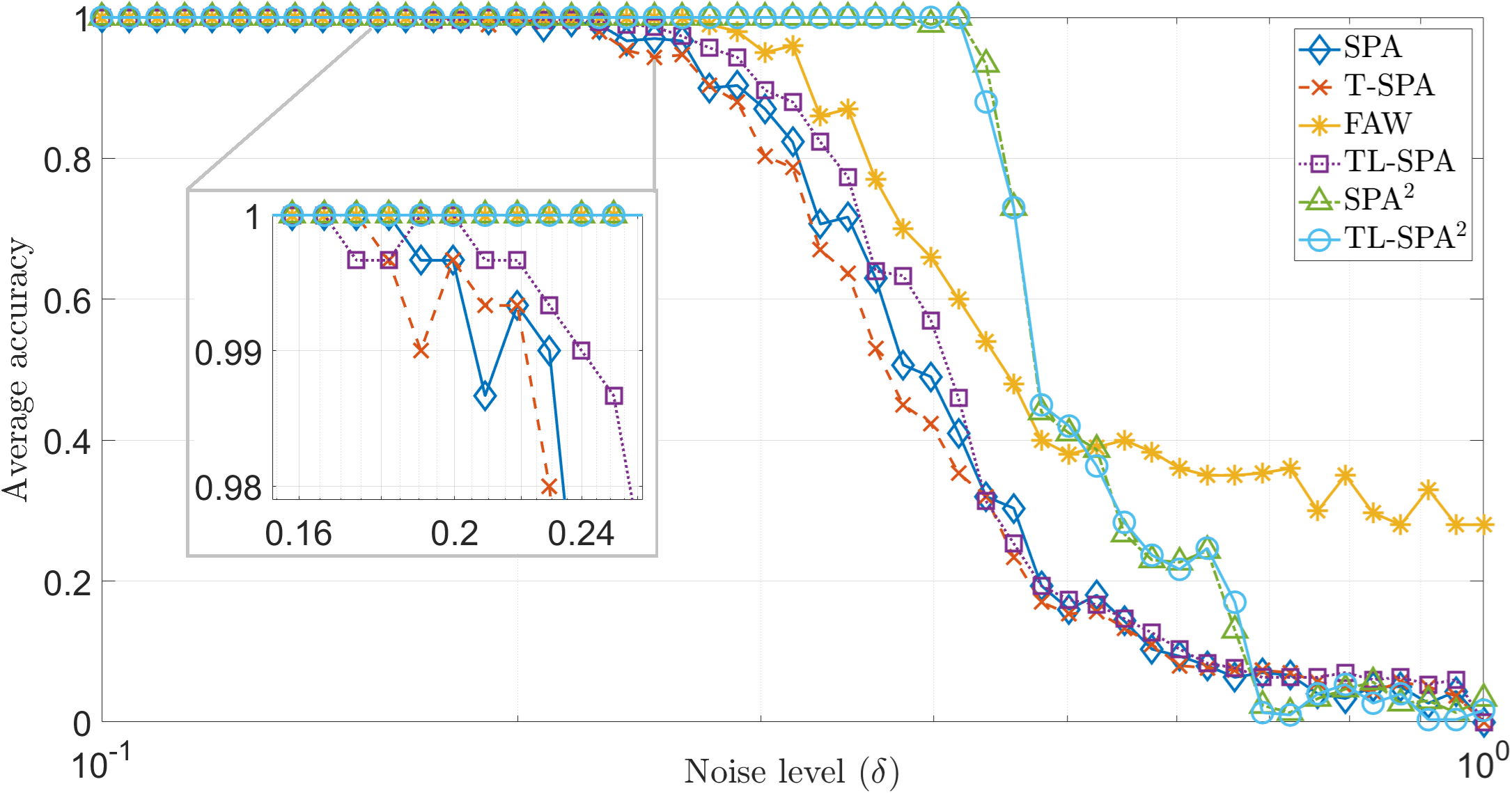}  \vspace{0.2cm} \\ 
Exp.~4:  Rank-deficient middle point experiment \\ 
    \includegraphics[width=0.95\textwidth]{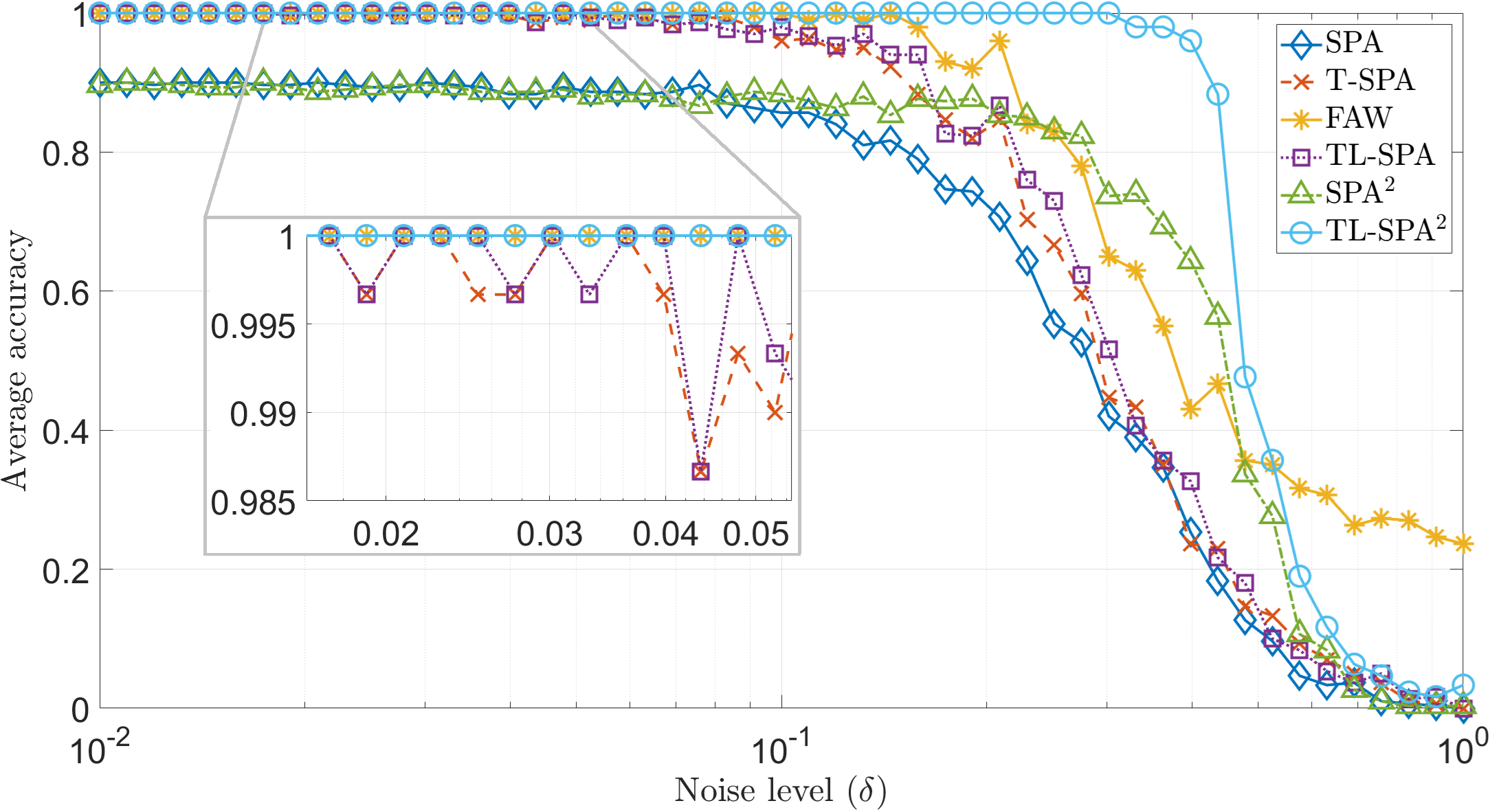} 
\end{tabular}
\caption{Accuracy vs.\ noise level for the middle point experiments.  \label{fig:Middlepoint}}  
\end{center}
\end{figure*} 

Table~\ref{table:SPAlikesynt} reports the robustness of each algorithm in each experimental setting. The robustness is defined as the largest noise level for which the corresponding algorithm recovered exactly all indices for all matrices.  
\begin{table}[h!] 
\caption{Robustness comparison of SPA-like algorithms: largest value of the noise level $\delta$ for which each algorithm recovers all indices correctly for the 30 generated matrices.  
The best result is highlighted in bold.
\label{table:SPAlikesynt}} 
\begin{center}  
\begin{tabular}{c|cccccc} 
  &  SPA        &  T-SPA      &  FAW        &  TL-SPA     &  SPA$^2$    &  TL-SPA$^2$\\ \hline 
 Exp.\ 1 & \textbf{0.136}       & \textbf{0.136}       & \textbf{0.136}       & \textbf{0.136}       & \textbf{0.136}       & \textbf{0.136}       \\ 
 Exp.\ 2 & 5.03$\ 10^{-6}$    & 
 $\mathbf{1.93 \ 10^{-5}}$    & $\mathbf{1.93 \ 10^{-5}}$   & $\mathbf{1.93 \ 10^{-5}}$    & $\mathbf{1.93 \ 10^{-5}}$   &$\mathbf{1.93 \ 10^{-5}}$   \\ 
 Exp.\ 3 & 0.182       & 0.174       & 0.263       & 0.166       & 0.380       & \textbf{0.417}       \\ 
 Exp.\ 4 & 0       & 0.017       & 0.100       & 0.017       & 0       & \textbf{0.302}   
\end{tabular} 
\end{center}
\end{table} 

\newpage 

We observe the following: 
\begin{itemize}
    \item For the well-conditioned Dirichlet experiment  (Exp.\ 1), all SPA variants perform similarly and have the same robustness. The main reason is that the conditioning of $W$ and its translated variant, $W_t$, are close to each other. 
    This is a similar observation as in \cite[Fig~7.6]{gillis2020nonnegative}. 
    
    \item For the ill-conditioned Dirichlet experiment  (Exp.\ 2), SPA and SPA$^2$ perform slightly worse than the other algorithms. The reason is that $W$ has a very small eigenvalue ($10^{-6}$) that is slightly better handled via  translation at the first step. However, the gain of T-SPA, FAW, TL-SPA and TL-SPA$^2$ is not significant because the second singular value is also small (=$4.6 10^{-6}$ due to the logspaced singular values) so that the conditioning of   $W_t$ is not significantly smaller than that of $W$.

    \item For the well-conditioned middle point experiment (Exp.\ 3), there is a clear hierarchy: SPA and T-SPA perform similarly (since the matrix is well-conditioned, translation does not bring much improvements), 
    FAW is more robust than TL-SPA but its accuracy decays more rapidly, SPA$^2$ outperforms FAW (as already noted in \cite[Fig~7.7]{gillis2020nonnegative}), while TL-SPA$^2$ and SPA$^2$ perform similarly although TL-SPA$^2$ is more robust (see Table~\ref{table:SPAlikesynt}).

    \item For the rank-deficient middle point experiment (Exp.\ 4), SPA and SPA$^2$ cannot extract more than 9 indices, since $m=9 < r=10$, and hence cannot have an accuracy larger than $90\%$. This is because $W$ has only $r-1$ rows and hence as an infinite conditioning (the $r$th singular value is equal to zero), 
    while $W_t$ is well conditioned (since the $(r-1)$th singular value of $W$ is not small, with high probability). 
    For the other algorithms, we observe a similar hierarchy as for the well-conditioned middle point experiment (Exp.\ 3), except that our newly proposed TL-SPA$^2$ significantly outperforms the other variants (robustness of 0.302 vs.\ 0.100 for the second best, FAW).  
\end{itemize}  
These experiments illustrate that, in adversarial and rank-deficient or ill-conditioned settings, it is key to use enhanced variants of SPA. In particular, the proposed TL-SPA$^2$ has shown to outperform all other variants in such cases. 

\revise{
To further validate this observation, Table~\ref{table:cond} reports the conditioning of the matrix $W$ after the preprocessing performed by the different algorithms, namely: 
\begin{itemize}
    \item For SPA: $\mathcal K(W)$, since SPA applies directly to the data. 
    
\item For T-SPA and FAW: $\frac{K(W_t)}{\sigma_{r-1}(W_t)}$, where $W_t$ is the matrix $W$ translated by the column of $X$ with the largest $\ell_2$ norm; see Section~\ref{sec:draw1}. 

\item For TL-SPA: $\mathcal K(W_\ell)$, where $W_\ell = \begin{pmatrix}
	 W - v e^\top  \\
	c \; e^\top  
\end{pmatrix}$ with $v = Xe/n$ and $c$ given by~\eqref{eq:cvalue}.  

\item For SPA$^2$: $\mathcal K\left( (X(:,\mathcal J))^\dagger W \right)$,  
where $\mathcal J$ is obtained by SPA. 

\item For TL-SPA$^2$: $\mathcal K\left( (\tilde{X}(:,\mathcal J))^\dagger W_\ell \right)$, where $\mathcal J$ is obtained by TL-SPA, and $\tilde{X} = \begin{pmatrix}
	 X - v e^\top  \\
	c \; e^\top  
\end{pmatrix}$ is the translated and lifted version of $X$. 
\end{itemize}

We observe the following: 
\begin{itemize}
    \item Except for Exp.~2 and $\delta = 0.7$, it is always TL-SPA or TL-SPA$^2$ that lead to the lowest condition numbers.  

    \item For Exp.~1, where $W$ is well conditioned, the variants using the left-inverse preconditioning (that is, TL-SPA$^2$, and to a lesser extent SPA$^2$)  lead to high condition numbers for high noise levels (for $\delta \geq 0.46$). 
    However, this is not really worrisome since at these noise levels, the performance of all algorithms degrade significantly (robustness is lost for $\delta < 0.136$, see Table~\ref{table:SPAlikesynt}). 

    \item Exp.~2, where $W$ is ill conditioned, TL-SPA$^2$ is very effective at reducing the conditioning, especially for low noise levels, which explains its robustness results from Table~\ref{table:SPAlikesynt}. 

    \item Exp.~4, where $W$ is rank deficient, TL-SPA$^2$ is very effective in reducing the conditioning, when compared to T-SPA and TL-SPA, again confirming the results from Table~\ref{table:SPAlikesynt}. 
    
\end{itemize}

\begin{table}[h!] 
\revise{
\caption{\revise{Average condition number, and standard deviation, of the matrix $W$ after preprocessing using the different algorithms, for the 30 generated matrices (same settings as in Table~\ref{table:SPAlikesynt}). The first column reports the noise level $\delta$.  
The best result is highlighted in bold.}
\label{table:cond}} 
\begin{center}  
\begin{tabular}{c|ccccc} 
        Exp. 1  &  SPA &  T-SPA \& FAW        &  TL-SPA               &  SPA$^2$              &  TL-SPA$^2$          \\ \hline  
 0.10       & $ 3.7 \pm 0.4 $     & $ 2.7 \pm 0.3 $     & $ 4.5 \pm 1.4 $     & $ \mathbf{1.2} \pm 0.0 $     & $ \mathbf{1.2} \pm 0.0 $     \\ 
 0.28       & $ 3.9 \pm 0.4 $     & $ 2.8 \pm 0.4 $     & $ 1.9 \pm 0.2 $     & $ 1.8 \pm 0.2 $     & $ \mathbf{1.7} \pm 0.2 $     \\ 
 0.46       & $ 3.7 \pm 0.4 $     & $ 2.9 \pm 0.4 $     & $ \mathbf{2.0} \pm 0.2 $     & $ 23.6 \pm 25.8 $     & $ 22.6 \pm 52.8 $     \\ 
 0.64       & $ 3.7 \pm 0.5 $     & $ 3.1 \pm 0.4 $     & $ \mathbf{2.1} \pm 0.2 $     & $ 75.4 \pm 241.1 $     & $ 212.4 \pm 775.5 $     \\ 
 0.82       & $ 3.8 \pm 0.3 $     & $ 3.6 \pm 0.4 $     & $ \mathbf{2.5} \pm 0.2 $     & $ 58.1 \pm 78.8 $     & $ 135.5 \pm 501.8 $     \\ 
 1       & $ 3.8 \pm 0.3 $     & $ 4.0 \pm 0.5 $     & $ \mathbf{2.6} \pm 0.2 $     & $ 49.1 \pm 60.1 $     & $ 283.6 \pm 621.0 $     \\ 
\hline 
 Exp. 2  &  SPA &  T-SPA \& FAW        &  TL-SPA               &  SPA$^2$              &  TL-SPA$^2$          \\ \hline  
 $10^{-6}$    & $ 54.7 \pm 5.55\cdot 10^4 $       & $ 20.4 \pm 1.93\cdot 10^4 $       & $ 11.9 \pm 1.20\cdot 10^4 $       & $ {1.1} \pm 0.0 $     & $ \mathbf{1.0} \pm 0.0 $     \\ 
 $1.48 \ 10^{-5}$    & $ 55.9 \pm 7.69\cdot 10^4 $       & $ 19.6 \pm 2.03\cdot 10^4 $       & $ 12.2 \pm 1.65\cdot 10^4 $       & $ {6.6} \pm 2.8 $     & $ \mathbf{1.3} \pm 0.1 $     \\ 
 $2.18 \ 10^{-4}$    & $ 59.4 \pm 7.68\cdot 10^4 $       & $ 51.2 \pm 5.66\cdot 10^3 $       & $ 13.0 \pm 1.75\cdot 10^4 $       & $ {802.3} \pm 1157.2 $     & $ \mathbf{118.1} \pm 122.7 $     \\ 
 0.003       & $ 58.5 \pm 7.81\cdot 10^4 $       & $ 48.4 \pm 4.63\cdot 10^3 $       & $ 12.8 \pm 1.74\cdot 10^4 $       & $ {1.13} \pm 1.51\cdot 10^4 $       & $ \mathbf{4.57} \pm 8.37\cdot 10^3 $       \\ 
 0.047       & $ 58.5 \pm 7.39\cdot 10^4 $       & $ {49.6} \pm 4.59 \cdot 10^3 $       & $ 12.7 \pm 1.56\cdot 10^4 $       & $ 8.01 \pm 30.9\cdot 10^5 $       & $ \mathbf{2.19} \pm 3.13\cdot 10^4 $       \\ 
 0.7       & $ 57.4 \pm 6.06\cdot 10^4 $       & $ \mathbf{49.1} \pm 4.68 \cdot 10^3 $       & $ 12.7 \pm 1.26\cdot 10^4 $       & $ 2.21 \pm 5.67 \cdot 10^6 $       & $ 4.31 \pm 5.94\cdot 10^5 $       \\ 
\hline 
 Exp. 3  &  SPA &  T-SPA \& FAW        &  TL-SPA               &  SPA$^2$              &  TL-SPA$^2$          \\ \hline  
 0.10       & $ 3.9 \pm 0.4 $     & $ 2.7 \pm 0.3 $     & $ 2.2 \pm 0.2 $     & $ \mathbf{1.0} \pm 0.0 $     & $ \mathbf{1.0} \pm 0.0 $     \\ 
 0.16       & $ 3.7 \pm 0.4 $     & $ 2.6 \pm 0.3 $     & $ 2.2 \pm 0.2 $     & $ \mathbf{1.0} \pm 0.0 $     & $ \mathbf{1.0} \pm 0.0 $     \\ 
 0.25       & $ 3.7 \pm 0.4 $     & $ 2.6 \pm 0.3 $     & $ 2.2 \pm 0.2 $     & $ 1.6 \pm 0.7 $     & $ \mathbf{1.2} \pm 0.5 $     \\ 
 0.40       & $ 3.7 \pm 0.4 $     & $ 2.6 \pm 0.3 $     & $ \mathbf{2.2} \pm 0.2 $     & $ 2.7 \pm 0.2 $     & $ 2.7 \pm 0.2 $     \\ 
 0.63       & $ 3.8 \pm 0.6 $     & $ 3.0 \pm 0.4 $     & $ \mathbf{2.3} \pm 0.3 $     & $ 3.2 \pm 0.4 $     & $ 3.1 \pm 0.3 $     \\ 
 1       & $ 3.9 \pm 0.4 $     & $ 3.3 \pm 0.4 $     & $ \mathbf{2.3} \pm 0.2 $     & $ 3.5 \pm 0.4 $     & $ 3.6 \pm 0.4 $     \\ 
\hline 
 Exp. 4  &  SPA &  T-SPA \& FAW        &  TL-SPA               &  SPA$^2$              &  TL-SPA$^2$          \\ \hline  
 0.01       & $ \infty $       & $ 72.0 \pm 95.5 $     & $ 57.7 \pm 73.8 $     & $ \infty $       & $ \mathbf{1.0} \pm 0.0 $     \\ 
 0.025       & $ \infty $       & $ 53.1 \pm 72.0 $     & $ 42.6 \pm 57.6 $     & $ \infty $       & $ \mathbf{1.0} \pm 0.0 $     \\ 
 0.063       & $ \infty $       & $ 7.65\pm 38.3\cdot 10^3 $       & $ 6.51 \pm 32.8\cdot 10^3 $       & $ \infty $       & $ \mathbf{1.2} \pm 0.5 $     \\ 
 0.16       & $ \infty $       & $ 407 \pm 1610 $     & $ 324 \pm 1252 $     & $ \infty $       & $ \mathbf{1.8} \pm 0.7 $     \\ 
 0.40       & $ \infty $       & $ 56.5 \pm 94.6 $     & $ 41.4 \pm 72.1 $     & $ \infty $       & $ \mathbf{2.9} \pm 0.4 $     \\ 
 1      & $ \infty $       & $ 195 \pm 600 $     & $ 129 \pm 416 $     & $ \infty $       & $ \mathbf{3.5} \pm 0.4 $     \\  
\end{tabular} 
\end{center}
}
\end{table}  
}
 
\revise{
\begin{remark}[Sensitivity of TL-SPA to the lift parameter] In Section~\ref{sec:preprocess}, we proposed a lift parameter for TL-SPA based on Lemma~\ref{lem:lemmaCvalue}; see~\eqref{eq:cvalue}. 
This is a heuristic choice with no theoretical guarantees since $W$ is unknown in practice.   
However, we have observed that TL-SPA is not too sensitive to the choice of this parameter. 
We have rerun TL-SPA with the lift parameter $c$ defined in~\eqref{eq:cvalue} multiplied by a factor $\alpha \in \{ 0.1, 0.5, 2, 10\}$. 
Table~\ref{table:TLspacond} reports the average condition number, $\mathcal K(W_\ell)$, for all experiments. 
Except when $\alpha = 10$ that leads to significantly worse results, the other condition numbers are very close to each other.  
\begin{table}[h!] 
\revise{
\caption{\revise{Average conditioning $\mathcal K(W_\ell)$ using different shifts $c$ defined in~\eqref{eq:cvalue} multiplied by a factor $\alpha \in \{ 0.1, 0.5, 2, 10\}$. Same settings as in Table~\ref{table:cond} but the results are averaged over all noise levels. The best result is highlighted in bold.}
\label{table:TLspacond}} 
\begin{center}  
\begin{tabular}{c|ccccc} 
              &  $\alpha = 1/10$ &  $\alpha = 1/2$ &  $\alpha = 1  $ &  $\alpha = 2  $ &  $\alpha = 10 $\\ \hline  
Exp. 1  
& $ 6.3 \pm 6.4 $     
& $ 3.0 \pm 2.6 $     
& $ \mathbf{2.6} \pm 1.1 $     
& $ 3.0 \pm 0.9 $      
& $ 10.6 \pm 6.2 $     \\
Exp. 2  
& $ \mathbf{12.4} \pm 1.60\cdot 10^4 $       
& $ \mathbf{12.4} \pm 1.58\cdot 10^4 $       
& $ 12.5 \pm 1.57\cdot 10^4 $        
& $ 13.0 \pm 1.54\cdot 10^4 $        
& $ 22.9 \pm 2.44\cdot 10^4 $       \\ 
Exp. 3 
& $ 4.2 \pm 0.3 $      
& $ \mathbf{1.9} \pm 0.2 $      
& $ 2.2 \pm 0.2 $     
& $ 3.2 \pm 0.4 $     
& $ 13.6 \pm 1.6 $     \\ 
Exp. 4 
& $ \mathbf{1.03} \pm 11.8 \cdot 10^3 $       
& $ 1.07 \pm 12.2\cdot 10^3 $        
& $ 1.18 \pm 13.4\cdot 10^3 $        
& $ 1.56 \pm 17.6\cdot 10^3 $        
& $ 5.92 \pm 66.4\cdot 10^3 $       
 \end{tabular} 
\end{center}
}
\end{table} 
\end{remark}
}

\newpage

\section{Conclusion and further work} 

In this paper, we revisited the robustness to noise of SPA and variants. 
We improved previous error bounds for the first step of SPA and when $r=2$ (Theorems~\ref{theo:single_step_SPA} and~\ref{theo:rank2_SPA_Q}), and  
for the first two steps and when $r=3$ for the translated variant of SPA, T-SPA (Theorems~\ref{theo:firsttwo_T-SPA} and~\ref{theo:rank3_T-SPA}).  
We then proved that, for $r \geq 3$, the original robustness result of SPA~\cite[Theorem~3]{gillis2013fast} is tight (Theorem~\ref{theo:tightSPA}), an important open question in the literature~\cite[p.~231]{gillis2020nonnegative}, using a well-constructed family of matrices that achieve the worst-case error bound. 
We also proved the tightness of the error bounds for two preconditioned variants of SPA, namely SPA preconditioned with itself (SPA$^2$, Theorem~\ref{theo:tightSPA2}) and SPA preconditioned with minimum-volume ellipsoid (MVE-SPA, Theorem~\ref{theo:tightMVESPA}). 
Finally, we proposed a new SPA variant that uses a translation and lifting of the data points, namely TL-SPA, and that allows us to improve the robustness of SPA by reducing the condition number of $W$. We illustrated this on numerical experiments. 

Further work include the following: 
\begin{itemize}
    \item TL-SPA does not have provable  robustness guarantees as the choice of the translation and lifting may influence the conditioning of $W$ in an unexpected way. Providing such guarantees would be important to better understand its behavior. This could also lead to new, more clever, ways to preprocess the input data matrix $X$. 

    \item Tightness of other separable SSMF algorithms could also be studied, e.g., the ones  based on convex relaxations (see Section~\ref{sec:draw2}). 
    Maybe our counter examples for SPA, SPA$^2$ and MVE-SPA could be the basis for such an analysis. 

\item Comparing the performance of the SPA variants for real-world applications would help understand which variant to use in which situation.

\end{itemize}

\revise{
\section*{Acknowledgments}  We are grateful to the anonymous reviewers who
carefully read the manuscript, their feedback helped us improve our paper.  
}

\appendix

\section{Proofs} \label{app:proofs}

\subsection{Proof of Lemma~\ref{lem:median_point}} \label{app:proofsLem1}

\begin{proof} 
  Given a couple of distinct indices $i,j$, let $W_{i,j} = [w_i, w_j]$. We have  
    \[
      \left\| \frac{w_i+w_j}2 \right\|^2  =  \frac 12\|w_i\|^2 + \frac 12\|w_j\|^2- \frac 14\|w_i-w_j\|^2\le K(W_{i,j})^2 - \frac 14\|W_{i,j}(e_1-e_2)\|^2 
      \le K(W_{i,j})^2 - \frac 12\sigma_2(W_{i,j})^2. 
    \]
     This implies 
    \begin{align*}
           \left\| \frac{w_i+w_j}2 \right\|&\le \sqrt{ K(W_{i,j})^2 - \frac 12\sigma_2(W_{i,j})^2}
    \le  \sqrt{ K(W_{i,j})^2 - \frac 12\sigma_2(W_{i,j})^2 + \frac{\sigma_2(W_{i,j})^4}{16K(W_{i,j})^2}}\\& = K(W_{i,j}) - \frac{\sigma_2(W_{i,j})^2}{4K(W_{i,j})}=
     K(W_{i,j}) - \frac{\sigma_2(W_{i,j})}{4\mathcal K(W_{i,j})}. 
    \end{align*}
    By the theorem of interlacing of singular values $\sigma_2(W_{i,j})\ge \sigma_r(W)$ and $\mathcal K(W_{i,j})\le \mathcal K(W)$, so 
    \begin{align*}
         \left\| \frac{w_i+w_j}2 \right\|&\le    K(W_{i,j}) - \frac{\sigma_2(W_{i,j})}{4\mathcal K(W_{i,j})}\le 
            K(W_{i,j}) - \frac{\sigma_r(W)}{4\mathcal K(W)}
         \le \max\{\|w_i\|,\|w_j\|\} -2\varepsilon. 
    \end{align*}
\end{proof}

\subsection{Proof of Lemma~\ref{lem:decomposition_convex}} \label{app:proofsLem2}

\begin{proof}[Proof of Lemma~\ref{lem:decomposition_convex}]
        Notice that since all $w_i$ and all $v_i$ with $i\ge 1$ are in $C$, then  $\conv(v_1,\dots,v_r,w_1,\dots,w_r)\cu C$. As a consequence, to prove $V\setminus C\cu \conv(v_0,w_1,\dots,w_r)$ we just need to show that 
    \[
    V = \conv(v_1,\dots,v_r,w_1,\dots,w_r) \cup \conv(v_0,w_1,\dots,w_r).
    \]
Let $\wt v_i := v_i-v_0$ and $\wt w_i := w_i - v_0 = \alpha_i\wt v_i$.  By the definition of convex hull, $v\in V$ if and only if there exist nonnegative $\lambda_0,\dots, \lambda_r$ that sum to one such that $\wt v:= v - v_0 =-v_0+ \sum_i \lambda_i v_i = \sum_{i\ge 1}\lambda_i\wt v_i$. 
If $\sum_{i\ge 1} \lambda_i/\alpha_i\le 1$, then
\[
\wt v = \sum_{i\ge 1}\frac{\lambda_i}{\alpha_i}\wt w_i\implies 
v = v_0\left(1-\sum_{i\ge 1}\frac{\lambda_i}{\alpha_i}\right) +\sum_{i\ge 1} \frac{\lambda_i}{\alpha_i} w_i\in \conv(v_0,w_1,\dots,w_r).
\]
Otherwise, $\sum_{i\ge 1} \lambda_i/\alpha_i> 1$, but since $\sum_{i\ge 1} \lambda_i = 1-\lambda_0\le 1$ then there exists $t\in [0,1)$ such that $\sum_{i\ge 1} \lambda_i/(1 + t(\alpha_i-1))= 1$.
Let now  $\wt\lambda_i:= \lambda_i/(1 + t(\alpha_i-1))$ for every $i\ge 1$, so that
 $\sum_{i\ge 1}\wt \lambda_i=1$ and
\[
t\wt\lambda_i \wt w_i + (1-t)\wt\lambda_i  \wt v_i = (1+(\alpha_i-1)t)\wt\lambda_i \wt v_i=\lambda_i\wt v_i.
\]
Therefore 
\[
\wt v = \sum_{i\ge 1}t\wt\lambda_i \wt w_i + (1-t)\wt\lambda_i  \wt v_i
\implies  v = \sum_{i\ge 1}t\wt\lambda_i  w_i + (1-t)\wt\lambda_i  v_i \in \conv(v_1,\dots,v_r,w_1,\dots,w_r).
\]
\end{proof}

\subsection{Proof of Lemma~\ref{lem:translated_least_sv}}  \label{app:proofsLem3} 

\begin{proof}
     Notice that $\sigma_{r-1}(\hat W) = \sigma_{r-1}(\wt W - ve^\top )$, where 
    $$
    \wt W = (w_1,\, w_2, \dots,w_{r-1},\,v) =
    W \begin{pmatrix}
         & & \lambda_1\\ 
            I_{r-1} & &\vdots \\
                 & & \lambda_{r-1}\\
        0_{1\times(r-1)}& & \lambda_r
    \end{pmatrix}= W (I + \wt \lambda e_r^\top ), 
    $$
 with $\wt \lambda_i = \lambda_i$ for $1\le i<r$ and $-\wt \lambda_r = 1-\lambda_r = \lambda_1+\dots+\lambda_{r-1}$.    Due to the interlacing property of singular values, we have that $\sigma_{r-1}(\wt W - ve^\top )\ge \sigma_{r}(\wt W)\ge \sigma_{r}( W)\sigma_{r}(I + \wt \lambda e_r^\top )$.  Moreover, due to Weyl's perturbation law,
 \[
 \sigma_r(I + \wt \lambda e_r^\top ) \ge 1 - \|\wt \lambda\| = 1 - \sqrt{\lambda_1^2 + \dots + \lambda_{r-1}^2 + (1-\lambda_r)^2} 
 \ge  1 - \sqrt{2(1-\lambda_r)^2}  \ge 
 1 - \frac{\sqrt{2}}2.
 \]
 This lets us conclude that $\sigma_{r-1}(\hat W)\ge \sigma_{r}( W)\frac{2-\sqrt 2}{2}$. For the second part, it is enough to notice that 
 $\|v\|\le K(W)$ since it belongs to $\conv(W)$, so $K(\hat W)\le 2K(W)$ is immediate.
\end{proof}

\subsection{Proof of Lemma~\ref{lem:first_step_TSPA}}  \label{app:proofsLem4} 

\begin{proof}
  Since  $c\le 1/8$, we can apply Theorem \ref{theo:single_step_SPA} and find that $\|x_1-w_1\| \le33\mathcal K(W)\ve$. The translation brings the matrix $X$ into   $\hat X$  as in the hypothesis. 
     From lemma \ref{lem:translated_least_sv}, we get that 
     \[ \sigma_{r-1}(\hat W) \ge \frac{2-\sqrt 2}{2}\sigma_r(W) > 0,
     \]  
so $\hat W\in \f R^{m\times(r-1)}$ is full rank. Notice also that 
\[
 \|r_1\| \le\|x_1-w_1\| + \|x_1-w\|\le
 34\mathcal K(W)\ve\le34c\sigma_r(W)\le 
 \frac{68c}{2-\sqrt 2} \sigma_{r-1}(\hat W)\le
 \frac 12 \sigma_{r-1}(\hat W).
 \]
Lastly, $K(\hat W) \le 2K(W)$ that implies
 \[
  \sigma_r(W) \mathcal K(W)^{-1} 
  =  \sigma_r(W)^2 K(W)^{-1}
  \le 
 \frac { 8}{\left(2-\sqrt 2   \right)^2} \sigma_{r-1}(\hat W)^2  K(\hat W)^{-1} = 
 \frac { 4}{3-2\sqrt 2} \sigma_{r-1}(\hat W) \mathcal  K(\hat W)^{-1} 
 \]
 and 
 \[
\mathcal K(\hat W) = \sigma_{r-1}(\hat W) ^{-1}K(\hat W)\le 
 \frac { 4}{2-\sqrt 2  }\sigma_{r}( W) ^{-1}K(W)= 
  \frac { 4}{2-\sqrt 2  }\mathcal K( W). 
 \]
\end{proof}

\bibliographystyle{spmpsci}
\bibliography{references}

\end{document}